\documentclass[11pt, leqno,twoside]{article}
\usepackage{amssymb}
\usepackage{amsmath}
\usepackage{amsthm}
\usepackage{amsfonts}
\usepackage{color}
 \usepackage{graphicx}

\usepackage[active]{srcltx}

\allowdisplaybreaks

\def\rr{{\mathbb R}}
\def\rn{{{\rr}^n}}

\def\zz{{\mathbb Z}}
\def\nn{{\mathbb N}}

\def\ch{{\mathcal H}}
\def\cc{{\mathbb C}}

\def\cs{{\mathcal S}}
\def\cx{{\mathcal X}}

\def\bd{{ \mathbb D }}
\def\cd{{ \mathcal D }}

\def\cg{{\mathcal G}}

\def\cq{{\mathcal Q}}
\def\cp{{\mathcal P}}
\def\cm{{\mathcal M}}

\def\cb{{\mathcal B}}
\def\ca{{\mathcal A}}
\def\cy{{\mathcal Y}}

\def\fz{\infty}
\def\az{\alpha}

\def\lz{\lambda}
\def\dz{\delta}

\def\ez{\epsilon}
\def\kz{\kappa}
\def\bz{\beta}
\def\gz{{\gamma}}

\def\vz{\varphi}

\def\sz{\sigma}

\def\wz{\widetilde}
\def\ls{\lesssim}
\def\gs{\gtrsim}
\def\r{\right}
\def\lf{\left}

\def\diam{{\mathop\mathrm{\,diam\,}}}
\def\aoti{{\mathop\mathrm {\,AOTI\,}}}
\def\supp{{\mathop\mathrm{\,supp\,}}}

\def\loc{{\mathop\mathrm{\,loc\,}}}

\def\lip{{\mathop\mathrm{\,Lip}}}

 \def\ocg{{\mathring{\cg}}}
\def\bmo{{\mathrm{\,BMO\,}}}

\def\bint{{\ifinner\rlap{\bf\kern.25em--}
\int\else\rlap{\bf\kern.45em--}\int\fi}\ignorespaces}
\def\dbint{\displaystyle\bint}
\def\bbint{{\ifinner\rlap{\bf\kern.25em--}
\hspace{0.078cm}\int\else\rlap{\bf\kern.45em--}\int\fi}\ignorespaces}

\newtheorem{thm}{Theorem}[section]
\newtheorem{lem}{Lemma}[section]
\newtheorem{prop}{Proposition}[section]
\newtheorem{rem}{Remark}[section]
\newtheorem{cor}{Corollary}[section]
\newtheorem{defn}{Definition}[section]

\numberwithin{equation}{section}

\begin{document}

\arraycolsep=1pt

\title{\Large\bf Pointwise Characterizations of Besov and  Triebel-Lizorkin Spaces and Quasiconformal Mappings
\footnotetext{\hspace{-0.35cm}
\noindent{2000 {\it Mathematics Subject Classification:}}
Primary: 30C65; Secondary: 42B35, 42B25, 46E35, 30L10
\endgraf  {\it Key words and phases:}
quasiconformal mapping, quasisymmetric mapping,
fractional Haj\l asz gradient,
Besov space, Haj\l asz-Besov space, Triebel-Lizorkin space,  Haj\l asz-Triebel-Lizorkin space,
grand Besov space, grand Triebel-Lizorkin space,
metric measure space
\endgraf Dachun Yang was supported by the National
Natural Science Foundation (Grant No. 10871025) of China. Pekka Koskela and
Yuan Zhou were supported by the Academy of Finland grants 120972, 131477.
}}
\author{Pekka Koskela, Dachun Yang and Yuan Zhou}
\date{ }
\maketitle

\begin{center}
\begin{minipage}{13.5cm}\small
{\noindent{\bf Abstract}\quad
In this paper, the authors characterize, in terms of pointwise inequalities, 
the classical Besov spaces $\dot B^s_{p,\,q}$ and Triebel-Lizorkin spaces $\dot F^s_{p,\,q}$
for all $s\in(0,\,1)$ and $p,\,q\in(n/(n+s),\,\infty],$ both in ${\mathbb R}^n$ and in
the metric measure spaces enjoying the doubling and reverse doubling
properties.
Applying this characterization, the authors prove that
quasiconformal mappings
preserve $\dot F^s_{n/s,\,q}$ on $\rn$ for all $s\in(0,\,1)$
and $q\in(n/(n+s),\,\infty]$.
A metric measure space version of the above morphism property is also
established. 
}
\end{minipage}
\end{center}

\medskip

\section{Introduction\label{s1}}

\hskip\parindent

We begin by recalling the metric definition of quasiconformal mappings
and the definition of quasisymmetric mappings; see \cite{tv}.
Let $(\cx,\,d_\cx)$ and $(\cy,\,d_\cy)$ be metric spaces and $f:\ \cx\to \cy$
be a homeomorphism.
If there exists $H\in(0,\,\fz)$ such that for all
 $x\in\cx$,
$$\limsup_{r\to 0}\frac{\sup\{d_\cy(f(x),\,f(y)):\ d_\cx(x,\,y)\le r\}}
{\inf\{d_\cy(f(x),\,f(y)):\ d_\cx(x,\,y)\ge r\}}\le H,$$
then $f$ is called {\it quasiconformal}.
Moreover, if there exists a homeomorphism $\eta: [0,\,\fz) \to [0,\,\fz)$ such that
for all $a,\,b,\,x\in\cx$ with $x\ne b$,
$$
 \frac{d_\cy(f(x),\,f(a))}{d_\cy(f(x),\,f(b))}\le  \eta\lf(\frac{d_\cx(x,\,a)}{d_\cx(x,\,b)}\r),
$$
then $f$ is called {\it $\eta$-quasisymmetric}, and sometimes, simply, quasisymmetric.
Every quasisymmetric mapping is quasiconformal,
but the converse is always not true; see, for example, \cite{hk98} and the
references therein.

Let $n>1$ and $\cx=\cy=\rr^n$ equipped with the usual Euclidean distance.
Then quasiconformality is equivalent with quasisymmetry and further with
the analytic conditions that the first order distributional partial derivatives
of $f$ are locally integrable and
$$|Df(x)|^n \le K J(x,f)$$
almost everywhere (assuming that $f$ is orientation preserving).
The well-known result that the Sobolev space
$\dot W^{1,\,n}$ is invariant under quasiconformal mappings on $\rn$
then comes as no surprise; see, for example,
\cite[Lemma 5.13]{k09}. By a function space being invariant under
quasiconformal mappings we mean that both $f$ and $f^{-1}$ induce
a bounded composition operator.
Reimann \cite{r74} proved that also $\bmo$ is quasiconformally invariant by
employing the reverse  H\"older inequalities of Gehring \cite{g73}
for the Jacobian of a quasiconformal mapping. Both the above two invariance
properties essentially characterize quasiconformal mappings \cite{r74}.
These results extend to the setting of Ahlfors regular
metric spaces that support a suitable Poincar\'e inequality \cite{hk98},
\cite{km98}.
There are some further function spaces whose quasiconformal invariance
follows from the above results.
First of all, the trace space of $\dot W^{1,\,n+1}(\rr^{n+1})$ is the homogeneous
Besov space $\dot \cb^{n/(n+1)}_{n+1} (\rr^{n})$; see Section \ref{s4} for the
definition. Because each quasiconformal
mapping of $\rr^{n }$ onto itself extends to a quasiconformal mapping of
$\rr^{n+1}$ onto itself \cite{TVa}, one concludes the invariance of
$\dot \cb^{n/(n+1)}_{n+1} (\rr^n)$ with a bit of additional work.
Further function spaces that are invariant under quasiconformal changes
of variable are obtained using interpolation. For this,
it is convenient to work with Triebel-Lizorkin spaces, whose definitions
will be given in Section \ref{s3}. Recall
that $\bmo(\rn)=\dot F^0_{\fz,\,2}(\rn),$
$\dot W^{1,\,n}(\rn)=\dot F^1_{n,\,2}(\rn),$ and
$\dot \cb^{n/(n+1)}_{n+1} (\rr^n) =\dot F^{n/(n+1)}_{n+1,\, n+1}
(\rr^n).$ By interpolation, one concludes that also the Triebel-Lizorkin
spaces $\dot F^s_{n/s,\,2}(\rn)$ are invariant for all $s\in(0,\,1)$ and so are
$\dot F^s_{n/s,\,q}(\rn)$ when $s\in(0,\,n/(n+1))$ and $q=2n/(n-(n-1)s)$ or when
$s\in(n/(n+1),\,1)$ and $q=2/((n-1)s-n+2).$ Notice that above the allowable values
of $q$ satisfy $2<q<n/s.$

Recently, Bourdon and Pajot \cite{bp03}
(see also \cite{b07}) proved a general result for quasisymmetric
mappings, which, in the setting of $\rn,$ shows that the Triebel-Lizorkin
space $\dot F^s_{n/s,\,n/s}(\rn)$
is quasiconformally invariant, for each $s\in(0,\,1).$
Notice that the norms of all the Triebel-Lizorkin spaces considered above
are conformally invariant: invariant under translations, rotations and scalings
of $\rn.$ It is then natural to inquire if all such Triebel-Lizorkin
spaces are quasiconformally invariant.

Our first result shows that this is essentially the case.

\begin{thm}\label{t1.1}
Let $n\ge 2,\ s\in (0,1)$ and $q\in(n/(n+s),\,\fz]$.
Then $\dot F^{s}_{n/s,\,q}(\rr^n)$ is invariant under quasiconformal
mappings of $\rr^n$.
\end{thm}

The assumption $q>n/(n+s)$ may well be superficial in Theorem \ref{t1.1}, because
of the way it appears in our estimates. Indeed,
the proof of the above theorem is rather indirect: we establish the
quasiconformal invariance of a full scale of spaces defined by
means of pointwise inequalities initiated in the work of Haj\l asz \cite{h96} and
verify that, for most of the associated parameters, these spaces are
Triebel-Lizorkin spaces. Let us introduce the necessary notation.

In what follows, we say that $(\cx,\, d,\,\mu)$ is a {\it metric measure space} if
$d$ is a metric on $\cx$ and $\mu$
 a regular Borel measure on $\cx$ such that all balls defined by $d$
have finite and positive measures.

\begin{defn}\label{d1.1}\rm
Let $(\cx,\,d,\,\mu)$ be a metric measure space.
Let $s\in(0,\,\fz)$ and $u$ be a measurable function on $\cx$.
A sequence of nonnegative measurable functions,
${\vec g}\equiv\{g_k\}_{k\in\zz}$, is called
a {\it fractional $s$-Haj\l asz gradient} of $u$ if there exists
 $E\subset\cx$ with $\mu(E)=0$ such that for all $k\in\zz$ and
$x,\, y\in\cx\setminus E$
satisfying $2^{-k-1}\le d(x,\,y)< 2^{-k}$,
 \begin{equation}\label{e1.1}
|u(x)-u(y)|\le [d(x,\,y)]^s[g_k(x)+g_k(y)].
\end{equation}
Denote by $\bd^s(u)$ the {\it collection of all fractional $s$-Haj\l asz
gradients of $u$}.
\end{defn}

In fact,
${\vec g}\equiv\{g_k\}_{k\in\zz}$
above is not really a gradient. One should view it, in the Euclidean
setting (at least when $g_k=g_j$ for all $k,j$), as a maximal function
of the usual gradient. Relying on this concept we now
introduce counterparts of Besov  and
Triebel-Lizorkin spaces.
For simplicity, we only deal here with the case $p\in(0,\,\fz);$ the remaining
case $p=\fz$ is given in Section \ref{s2}.
In what follows, for $p,\,q\in(0,\,\fz]$, we always write
$\|\{g_j\}_{j\in\zz}\|_{ \ell^q}\equiv\{\sum_{j\in\zz}|g_j|^q\}^{1/q}$ when $q<\fz$
and $\|\{g_j\}_{j\in\zz}\|_{ \ell^\fz}\equiv \sup_{j\in\zz}|g_j|$,
$$\|\{g_j\}_{j\in\zz}\|_{L^p(\cx,\,\ell^q)}\equiv \|\|\{g_j\}_{j\in\zz}\|_{ \ell^q}\|_{L^p(\cx )}$$
and
$$\|\{g_j\}_{j\in\zz}\|_{\ell^q(L^p(\cx))}\equiv \|\{\|g_j\|_{L^p(\cx )}\}_{j\in\zz}\|_{ \ell^q}.$$

\begin{defn}\rm\label{d1.2}
Let $(\cx,\,d,\,\mu)$ be a metric measure space,
$s,\,p\in (0,\fz)$  and $q\in(0,\,\fz]$.

(i) The {\it homogeneous Haj\l asz-Triebel-Lizorkin space}
$\dot M^s_{p,\,q}(\cx)$ is the
space of all measurable functions $u$ such that
$$\|u\|_{\dot M^s_{p,\,q} (\cx)}\equiv\inf_{\vec g\in \bd^s (u)}\lf\| \vec g
\r\|_{L^p(\cx,\,\ell^q)}<\fz.$$

(ii) The {\it homogeneous Haj\l asz-Besov space} $\dot N^s_{p,\,q}(\cx)$ is the
space of all measurable functions $u$ such that
$$\|u\|_{\dot N^s_{p,\,q}(\cx)}\equiv\inf_{\vec g\in \bd^s (u)}\lf\| \vec g
\r\|_{\ell^q(L^p(\cx))}<\fz. $$
\end{defn}

Some properties and useful characterizations of $\dot M^s_{p,\,q}(\cx)$ and $\dot N^s_{p,\,q}(\cx)$
are given in Section \ref{s2}. In particular,
denote by $\dot M^{s,\,p}(\cx)$ the Haj\l asz-Sobolev space as in
Definition \ref{d2.2}.
 Then $\dot M^{s,\,p}(\cx)=\dot M^{s}_{p,\,\fz}(\cx)$ for
$s,\,p\in(0,\,\fz)$ 
as proved in Proposition \ref{p2.1}.

\begin{thm}\label{t1.2} Let $n\in\nn$.

(i) If $s\in(0,\,1)$,  $p\in(n/(n+s),\,\fz)$ and $q \in(n/(n+s),\,\fz]$, then
$\dot M^s_{p,\,q} (\rn)=\dot F^s_{p,\,q}(\rn)$.

(ii) If $s\in(0,\,1)$, $p \in(n/(n+s),\,\fz)$ and  $ q \in(0,\,\fz]$,
then
$\dot N^s_{p,\,q} (\rn)=\dot B^s_{p,\,q}(\rn)$.
\end{thm}

The equivalences above are proven via grand Besov spaces
$\ca\dot B^s_{p,\,q}(\rn)$ and  grand Triebel-Lizorkin spaces
$\ca\dot F^s_{p,\,q}(\rn)$ defined
 in Definition \ref{d3.2} below; see Section \ref{s3}.
Theorem \ref{t1.2} and Theorem \ref{t3.2} give pointwise characterizations 
for Beosv and Triebel-Lizorkin spaces and have independent interest.
For predecessors of such results, see \cite{ks08,y03}. 

Relying on Theorem  \ref{t2.1}, Lemmas \ref{l2.1} and \ref{l2.3} below 
and several properties of quasiconformal mappings  we obtain  the following invariance property that, 
when combined with Theorem \ref{t1.2}, yields Theorem \ref{t1.1}; see Section \ref{s5}.

\begin{thm}\label{t1.3} Let $n\ge 2$,  $s\in (0,1]$ and $q\in(0,\,\fz]$.
Then
$\dot M^{s}_{n/s,\,q}(\rn)$ is invariant under quasiconformal mappings
of $\rn.$
\end{thm}

The conclusion of Theorem \ref{t1.3} was previously only known in the case $s=1$
and $q=\fz$; recall that $\dot M^1_{n,\,\fz}(\rn)=\dot M^{1,\,n}(\rn)=\dot W^{1,\,n}(\rn)$.

Our results above also extend to a class of metric measure spaces.
Indeed, let $(\cx,\,d,\,\mu)$ be a metric
measure space.
For any $x\in\cx$ and $r>0$, let $B(x,\,r)\equiv\{y\in\cx:\, d(x,\, y)<r\}$.
Recall that $(\cx, d, \mu)$ is
called an {\it RD-space} in \cite{hmy08} if there exist constants
$0<C_1\le1\le C_2$ and
$0<\kz\le n$ such that for all $x\in\cx$, $0<r<2\diam\cx $ and
$1\le\lz<2\diam \cx/r$,
\begin{equation}\label{e1.2}
C_1\lz^\kz\mu(B(x, r))
\le\mu(B(x, \lz r))\le C_2\lz^n\mu(B(x, r)),
\end{equation}
where and in what follows,
$\diam \cx\equiv\sup_{x,\,y\in \cx}d(x,y)$;
see \cite{hmy08}. In particular, if $\kz=n$, then $\cx$ is called an
{\it Ahlfors $n$-regular space}.
Moreover, $\cx$ is said to support a {\it weak $(1,\,n)$-Poincar\'e inequality} if
there exists a positive constant $C$ such that for all 
Lipschitz functions $u$,
$$\bint_B|u(x)-u_B|\,d\mu(x)\le C r
\lf\{\bint_{\lz B}[\lip(u(x))]^n\,d\mu(x)\r\}^{1/n}.$$

We then have a metric measure space version of Theorem \ref{t1.3} as follows.  

\begin{thm}\label{t1.4}
Assume that $\cx$ and $\cy$ are both Ahlfors $n$-regular spaces with $n>1$,
$\cx$ is proper and quasiconvex and supports a weak $(1,\,n)$-Poincar\'e inequality
and $\cy$ is linearly locally connected.
Let $f$ be a quasiconformal mapping from $\cx$ onto $\cy$, which maps
bounded sets into bounded sets.
Then for every $s\in (0,\,1]$,
and for all $q\in(0,\,\fz]$, $f$  induces an equivalence between
$\dot M^{s}_{n/s,\,q}(\cx)$
and $\dot M^{s}_{n/s,\,q}(\cy)$.
\end{thm}

The point is that, with the assumptions of Theorem \ref{t1.4}, 
the quasiconformal mapping $f$ is actually a quasisymmetric mapping  
and its volume derivative  satisfies a suitable reverse H\"older inequality 
(see \cite[Theorem 7.1]{hk98}, \cite{kz08}  and also Proposition \ref{p5.3} below), 
which allow us to extend the proof of Theorem \ref{t1.3} to this more general setting. 
In Theorem \ref{t1.4}, both $f$ and $f^{-1}$ act as composition operators. 

We also show, see Section \ref{s4}, that the spaces $M^{s}_{n/s,\,q}(\cx)$
and $\dot M^{s}_{n/s,\,q}(\cy)$ identify with suitable Triebel-Lizorkin
spaces and thus a version of the invariance of Triebel-Lizorkin
spaces follows. Moreover, let us comment that our approach recovers
the invariance of the Besov spaces considered by Bourdon and Pajot \cite{bp03}; 
see Theorem \ref{t5.1} below.

Finally, we state some {\it conventions}. Throughout the paper,
we denote by $C$ a positive
constant which is independent
of the main parameters, but which may vary from line to line.
Constants with subscripts, such as $C_0$, do not change
in different occurrences. The notation $A\ls B$ or $B\gs A$
means that $A\le CB$. If $A\ls B$ and $B\ls A$, we then
write $A\sim B$.
Denote by $\zz$ the set of integers, $\nn$ the set of positive integers and $\zz_+\equiv\nn\cup\{0\}$.
For $\az\in\rr$, denote
by $\lfloor \az\rfloor$ the maximal integer no more than $\az$.
For any locally integrable function $f$,
we denote by $\bbint_E f\,d\mu$ the average
of $f$ on $E$, namely, $\bbint_E f\,d\mu\equiv\frac 1{\mu(E)}\int_E f\,d\mu$.

\section{Some properties of $\dot M^s_{p,\,q}(\cx)$ and $\dot N^s_{p,\,q}(\cx)$ }\label{s2}

In this section, we establish some properties of $\dot M^s_{p,\,q}(\cx)$ and $\dot N^s_{p,\,q}(\cx)$,
including the equivalence between $\dot M^s_{p,\,\fz}(\cx)$ and the Haj\l asz-Sobolev space (see Proposition \ref{p2.1}),
several equivalent characterizations for $\dot M^s_{p,\,q}(\cx)$ and $\dot N^s_{p,\,q}(\cx)$ 
(see Theorems \ref{t2.1} and \ref{t2.2}),
and Poincar\'e-type inequalities for 
$\dot M^s_{p,\,q}(\cx)$ with $\cx=\rn$ (see Lemmas \ref{l2.1} and \ref{l2.3}).

First, we introduce the Haj\l asz-Besov and Haj\l asz-Triebel-Lizorkin spaces also in the case $p=\fz$
as follows.

\begin{defn}\rm\label{d2.1}
Let $(\cx,\,d,\,\mu)$ be a metric measure space,
$s\in (0,\fz)$ and $q\in(0,\,\fz]$.

(i) The homogeneous Haj\l asz-Triebel-Lizorkin space
$\dot M^s_{\fz,\,q}(\cx)$ is the
space of all measurable functions $u$ such that
$\|u\|_{\dot M^s_{\fz,\,q} (\cx)}<\fz$,
where when   $q<\fz$,
$$\|u\|_{\dot M^s_{\fz,\,q} (\cx)}\equiv
\inf_{\vec g\in \bd^s (u)}\sup_{k\in\zz}\sup_{x\in\cx}
\lf\{\sum_{j\ge k}\bint_{B(x,\,2^{-k})} [g_j(y)]^q\,d\mu(y)\r\}^{1/q}
 $$
and when $q=\fz$,
$\|u\|_{\dot M^s_{\fz,\,\fz}(\cx)}\equiv
\inf_{\vec g\in \bd^s (u)}
\|\vec g \|_{L^\fz(\cx,\,\ell^\fz)}.
 $

(ii) The homogeneous Haj\l asz-Besov space $\dot N^s_{\fz,\,q}(\cx)$ is the
space of all measurable functions $u$ such that
$$\|u\|_{\dot N^s_{\fz,\,q}(\cx)}\equiv\inf_{\vec g\in \bd^s (u)}\lf\| \vec g
\r\|_{\ell^q(L^\fz(\cx))}<\fz. $$
\end{defn}

Then, we recall the definition of a Haj\l asz-Sobolev space 
 \cite{h96,h03} (see also \cite{y03}  for a fractional version). 

Let $(\cx,\,d,\,\mu)$ be a metric measure space.
For every $s\in (0,\fz)$ and measurable function $u$ on $\cx$, a non-negative
function $g$ is called an {\it $s$-gradient} of $u$
if there exists a set $E\subset \cx$ with $\mu(E)=0$ such that for all
$x,\ y\in\cx\setminus E$,
\begin{equation} \label{e2.1}
|u(x)-u(y)|\le [d(x,y)]^s[g(x)+g(y)].
\end{equation}
Denote by $\cd^s(u)$ the {\it collection of all $ s$-gradients of $u$}.

\begin{defn}\label{d2.2}\rm
Let $s\in(0,\,\fz)$ and $p\in (0,\fz]$. Then the
{\it homogeneous  Haj\l asz-Sobolev space} $\dot M^{s,\,p}(\cx)$ is the
set of all measurable functions $u$ such that
 $$\|u\|_{\dot M^{s,\,p}(\cx)}\equiv\inf_{g\in \cd^s(u)}\|g\|_{L^p(\cx)}<\fz.$$
\end{defn}

 Haj\l asz-Sobolev spaces naturally relate to  
  Haj\l asz-Triebel-Lizorkin spaces as follows.

\begin{prop}\label{p2.1}
 If $s\in(0,\,\fz)$ and $p\in(0,\,\fz]$, then
$\dot M^{s }_{p,\,\fz}(\cx)=\dot M^{s,\,p}(\cx)$.
\end{prop}

\begin{proof}
Let $u\in\dot M^{s,\,p}(\cx)$ and $g\in\cd^s(u)$.
Taking $g_k\equiv g$,  we know that $\vec g=\{g_k\}_{k\in\zz}\in\bd^s(u)$
and $\|\vec g\|_{L^p(\cx,\,\ell^q)}=\|g\|_{L^p(\cx)}$,
which implies that $u\in \dot M^{s}_{p,\,\fz}(\cx)$ with
$\|u\|_{\dot M^{s}_{p,\,\fz}(\cx)}=\|u\|_{\dot M^{s,\,p}(\cx)}$.

Conversely, let $u\in\dot M^s_{p,\,\fz}(\cx)$ and $\vec g\equiv\{g_k\}_{k\in\zz}\in\bd^s(u)$.
Taking $g\equiv\sup_{k\in\zz} g_k$, we have that $\|\vec g\|_{L^p(\cx,\,\ell^q)}=\|g\|_{L^p(\cx)}$,
which implies that $u\in\dot M^{s,\,p}(\cx)$ with
$\|u\|_{\dot M^{s}_{p,\,\fz}(\cx)}=\|u\|_{\dot M^{s,\,p}(\cx)}$. This finishes the proof of Proposition \ref{p2.1}.
\end{proof}

Now, we introduce several useful variants of $\bd^{s}(u)$ to
characterize $\dot M^s_{p,\,q}(\cx)$ and $\dot N^s_{p,\,q}(\cx)$.
To this end, let $s\in(0,\,\fz)$ and $u$ be a measurable function on $\cx$.

For $N_1,\,N_2\in\zz_+$, denote by $\bd^{s,\,N_1,\,N_2}(u)$
the {\it collection of all the sequences of nonnegative measurable functions},
${\vec g}\equiv\{g_k\}_{k\in\zz}$, satisfying that there exists
 $E\subset\cx$ with $\mu(E)=0$ such that
 for all  $k\in\zz$ and $x,\, y\in\cx\setminus E$ with $2^{-k-1-N_1}\le d(x,y)< 2^{-k+N_2}$,
 \begin{equation*}
|u(x)-u(y)|\le [d(x,\,y)]^s[g_k(x)+g_k(y)].
\end{equation*}

For $\ez\in(0,\,s]$ and  $N\in\nn$,
denote by $\wz\bd^{s,\,\ez,\,N}(u)$ the {\it collection of all the sequences of nonnegative measurable functions},
${\vec g}\equiv\{g_k\}_{k\in\zz}$, satisfying that there exists
 $E\subset\cx$ with $\mu(E)=0$ such that
 for all $x,\, y\in\cx\setminus E$,
\begin{equation*}
|u(x)-u(y)|\le [d(x,\,y)]^{s-\ez} \sum_{k\in\zz}2^{-k\ez}[g_k(x)+g_k(y)]\chi_{[2^{-k-1-N},\,\fz)}(d(x,\,y)).
\end{equation*}

For $\ez\in(0,\,\fz)$ and $N\in\zz$, denote by $\overline \bd^{s,\,\ez,\,N}(u)$
the {\it collection of all the sequences of nonnegative measurable functions},
${\vec g}\equiv\{g_k\}_{k\in\zz}$, satisfying that there exists
 $E\subset\cx$ with $\mu(E)=0$ such that
 for all  $x,\, y\in\cx\setminus E$,
\begin{equation*}
|u(x)-u(y)|\le  [d(x,\,y)]^{s+\ez}\sum_{k\in\zz}2^{k\ez}[g_k(x)+g_k(y)]\chi_{(0,\,2^{-k-N})}(d(x,\,y)).
\end{equation*}

Then we have the following equivalent characterizations of $\dot M^s_{p,\,q}(\cx)$.

\begin{thm}\label{t2.1}

\noindent (I) Let  $s,\,p\in(0,\,\fz)$ and $q\in(0,\,\fz]$.
Then the following are equivalent:

(i) $u\in  \dot {M}^s_{p,\,q}(\cx) $;

(ii) for every pair of $N_1,\,N_2\in\zz_+$, $\inf_{\vec g\in \bd^{s,\,N_1,\,N_2} (u)}\lf\| \vec g  \r\|_{L^p(\cx,\,\ell^q) }<\fz$;

(iii) for every pair of $\ez_1\in(0,\,s]$ and $N_3\in\zz$, $\inf_{\vec g\in \wz\bd^{s,\,\ez_1,\,N_3} (u)}
\lf\| \vec g  \r\|_{L^p(\cx,\,\ell^q) }<\fz$;

(iv)  for every pair of $\ez_2\in(0,\,\fz)$ and $N_4\in\zz$, $\inf_{\vec g\in \overline \bd^{s,\,\ez_2,\,N_4} (u)}
\lf\| \vec g  \r\|_{L^p(\cx,\,\ell^q) }<\fz$.

\noindent Moreover, given $\ez_1, \,\ez_2 ,N_1,\,N_2,\,N_3$ and $N_4$ as above, for all $u\in  \dot {M}^s_{p,\,q}(\cx)$,
\begin{eqnarray*}
\|u\|_{\dot {M}^s_{p,\,q}(\cx)}&&\sim \inf_{\vec g\in \bd^{s,\,N_1,\,N_2} (u)}
\lf\| \vec g  \r\|_{L^p(\cx,\,\ell^q) }\\
&&\sim\inf_{\vec g\in \wz\bd^{s,\,\ez_1,\,N_3} (u)}
\lf\| \vec g  \r\|_{L^p(\cx,\,\ell^q) }\sim
\inf_{\vec g\in \overline \bd^{s,\,\ez_2,\,N_4} (u)}
\lf\| \vec g  \r\|_{L^p(\cx,\,\ell^q) },
\end{eqnarray*}
where the implicit constants are independent of $u$.

(II) Let $s\in(0,\,\fz)$ and $p,\,q\in(0,\,\fz]$.
Then the above statements still hold with $\dot {M}^s_{p,\,q}(\cx)$ and
$L^p(\cx,\,\ell^q)$ replaced by $\dot {N}^s_{p,\,q}(\cx)$ and
$\ell^q(L^p(\cx))$, respectively.
\end{thm}

\begin{proof}
We first prove that (i) implies (ii), (iii) and (iv).
Let $u$ be a measurable function and $\vec g\in \bd^{s} (u)$.
Then for every pair of $N_1,\,N_2\in\zz_+$,
setting $h_k\equiv \sum_{j=-N_2}^{N_1}g_{k+j}$ for $k\in\zz$,
we know that
$\vec h\equiv\{h_k\}_{k\in\zz}\in \bd^{s,\,N_1,\,N_2} (u)$.
For every pair of $\ez_1\in(0,\,s]$ and   $N_3\in\zz$,
taking $h_k\equiv 2^{N_3\ez}g_{k-N_3}$ for all $k\in\zz $,
we have
$\vec h\equiv\{h_k\}_{k\in\zz}\in \wz\bd^{s,\,\ez_1,\,N_3} (u)$.
For every pair of $\ez_2\in(0,\,\fz)$ and $N_4\in\zz$,
taking $h_k\equiv 2^{N_4\ez_2}g_{k-N_4}$ for all $k\in\zz$,
we have
$\vec h\equiv\{h_k\}_{k\in\zz}\in \overline\bd^{s,\,\ez_2,\,N_4} (u)$.
Then it is to easy to see that in all of the above cases, we have
$\| \vec h  \|_{L^p(\cx,\,\ell^q) }\ls \|u\|_{\dot {M}^s_{p,\,q}(\cx)}$
and
$\| \vec h  \|_{\ell^q(L^p(\cx))}\ls \|u\|_{\dot N^s_{p,\,q}(\cx)}$.
Thus, (i) implies (ii), (iii) and (iv).

Now we prove the converse.
Since $ \bd^{s,\,N_1,\,N_2} (u)\subset\bd^{s} (u)$, we  have that (ii) implies (i).

To show that (iii) implies (i), let $u$ be a measurable function and $\vec g\in \wz\bd^{s,\,\ez_1,\,N_3} (u)$.
For all $k\in\zz$, set $h_k\equiv \sum_{j=k-N_3}^\fz 2^{ (k-j+1)\ez_1}g_j$.
Then $\vec h\equiv\{h_k\}_{k\in\zz}\in \bd^{s} (u)$.

Moreover, if $p\in(0,\,\fz)$, by the H\"older inequality when $q\in(1,\,\fz)$
and the inequality
\begin{equation}\label{e2.2}
\lf(\sum_{i\in\zz} |a_i|\r)^q\le \sum_i |a_i|^q
\end{equation}
for $\{a_i\}_{i\in\zz}\subset\rr$
when $q\in(0,\,1]$,
we   have
\begin{eqnarray*}
\| \vec h   \|^q_{\ell^q}&&\ls
 \sum_{k\in\zz}\lf[\sum_{j=k-N_3}^\fz 2^{ (k-j+1)\ez_1}g_j\r]^q \ls \sum_{k\in\zz} \sum_{j=k-N_3}^\fz 2^{-(j-k)q\ez_1/2}[g_j]^q
\ls \| \vec g  \|^q_{\ell^q},
\end{eqnarray*}
which gives that $\| \vec h   \|_{L^p(\cx,\,\ell^q) }\ls\|u\|_{\dot M^s_{p,\,q}(\cx)}$.
This also holds when $q=\fz$, as seen  with a slight modification.

On the other hand,  by the H\"older inequality when $p\in(1,\,\fz)$ and the inequality
\eqref{e2.2} with $q=p$ when $p\in(0,\,1]$, we have
\begin{eqnarray*}
\|\vec  h   \|^q_{\ell^q(L^p(\cx))}&&\ls
 \sum_{k\in\zz}\lf(\int_\cx\lf[\sum_{j=k-N_3}^\fz 2^{-(j-k-1)\ez_1}g_j(y)\r]^p\,d\mu(y)\r)^{q/p}\\
&& \ls \sum_{k\in\zz}\lf(\sum_{j=k-N_3}^\fz 2^{-(j-k)p\ez_1/2}\int_\cx\lf[g_j(y)\r]^p\,d\mu(y)\r)^{q/p}.
\end{eqnarray*}
Applying the H\"older inequality when $p/q\in(1,\,\fz)$ and the inequality \eqref{e2.2} with power $q/p$ instead of $q$
again when $p/q\in(0,\,1]$, we further have
\begin{eqnarray*}
\|\vec  h   \|^q_{\ell^q(L^p(\cx))}
&& \ls \sum_{k\in\zz}\sum_{j=k-N_3}^\fz 2^{-(j-k)q\ez_1/4}\lf(\int_\cx\lf[g_j(y)\r]^p\,d\mu(y)\r)^{q/p}\ls \| \vec g  \|^q_{\ell^q(L^p(\cx))},
\end{eqnarray*}
 which gives that $\| \vec h   \|_{\ell^q(L^p(\cx)) }\ls\|u\|_{\dot N^s_{p,\,q}(\cx)}$.
This also holds when $p=\fz$ or $q=\fz$, as easily seen.

To prove that (iv) implies (i), let $u$ be a measurable function and $\vec g\in \overline\bd^{s,\,\ez_2,\,N_4} (u)$.
For every $k\in\zz$, set $h_k\equiv \sum_{j=-\fz}^{ k-N_4}2^{(j-k)\ez_2}g_j$.
Then $\vec h\equiv\{h_k\}_{k\in\zz}\in \bd^{s} (u)$.

Moreover, if $p\in(0,\,\fz)$, by the H\"older inequality when $q\in(1,\,\fz)$ and
the inequality \eqref{e2.2}
when $q\in(0,\,1]$,
we  have
\begin{eqnarray*}
\| \vec h \|^q_{\ell^q}&&\ls
 \sum_{k\in\zz}\lf[\sum_{j=-\fz}^{k-N_4}  2^{(j -k)\ez_2}g_j\r]^q \ls \sum_k \sum_{j=-\fz}^{k-N_4} 2^{(j-k)q\ez_2/2}[g_j]^q
\ls \| \vec g  \|^q_{\ell^q},
\end{eqnarray*}
which gives that $\|\vec h\|_{L^p(\cx,\,\ell^q) }\ls\|u\|_{\dot M^s_{p,\,q}(\cx)}$.
This also extends to the case $q=\fz$.

Similarly, one can prove that $\| \vec h   \|_{\ell^q(L^p(\cx)) }\ls\|u\|_{\dot N^s_{p,\,q}(\cx)}$,
but we omit the details. 
This finishes the proof of Theorem \ref{t2.1}.
\end{proof}

\begin{thm}\label{t2.2}
Let  $s\in(0,\,1]$ and $q\in(0,\,\fz)$.
Then the following are equivalent:

(i) $u\in  \dot {M}^s_{\fz,\,q}(\cx) $;

(ii) for every $N_2\in\zz_+$,
\begin{equation}\label{e2.3}
\inf_{\vec g\in \bd^{s,\,0,\,N_2} (u)}\sup_{k\in\zz}\sup_{x\in\cx}
\lf(\sum_{j\ge k}\bint_{B(x,\,2^{-k})} [g_j(y)]^q\,d\mu(y)\r)^{1/q}
<\fz;
\end{equation}

(iii) for every pair of $\ez\in(0,\,s]$ and   $N_3\in\zz\setminus\nn$,
\begin{equation}\label{e2.4}
\inf_{\vec g\in \wz\bd^{s,\,\ez,\,N_3} (u)}\sup_{k\in\zz}\sup_{x\in\cx}
\lf(\sum_{j\ge k}\bint_{B(x,\,2^{-k})} [g_j(y)]^q\,d\mu(y)\r)^{1/q}
<\fz.
\end{equation}
\noindent Moreover, given $\ez$, $ N_2 $ and $N_3$ as above,
  for all $u\in  \dot {M}^s_{\fz,\,q}(\cx)$,
$\|u\|_{\dot {M}^s_{\fz,\,q}(\cx)}$ is equivalent to the given quantity.
\end{thm}
\begin{proof}
We first prove that (i) implies (ii) and (iii).
Let $u$ be a measurable function and $\vec g\in \bd^{s} (u)$.
Then for every  $N_2\in\zz_+$,
setting $h_k\equiv \sum_{j=-N_2}^{0}g_{k+j}$ for all $k\in\zz$,
we know that
$\vec h\equiv\{h_k\}_{k\in\zz}\in \bd^{s,\,0,\,N_2} (u)$.
For every pair of $\ez\in(0,\,s]$ and $N_3\in\zz\setminus\nn$,
taking $h_k\equiv 2^{N_3\ez}g_{k-N_3}$ for all $k\in\zz $,
we have
$\vec h\equiv\{h_k\}_{k\in\zz}\in \wz\bd^{s,\,\ez,\,N_3} (u)$.
Then it is to easy to see that in both cases,
\begin{equation*}
 \sum_{j\ge k}\bint_{B(x,\,2^{-k})} [h_j(y)]^q\,d\mu(y)
\ls \sum_{j\ge k}\bint_{B(x,\,2^{-k})} [g_j(y)]^q\,d\mu(y).
\end{equation*}
Thus, (i) implies (ii)  and (iii).

Conversely,  since
$\bd^{s,\,0,\,N_2} (u)\subset \bd^{s} (u)$, we have that (ii) implies (i).

To prove that (iii)  implies  (i), let $u$ be a measurable function and $\vec g\in \wz\bd^{s,\,\ez,\,N_3} (u)$.
For all $k\in\zz$, set $h_k\equiv \sum_{j=k-N_3}^\fz 2^{(k-j+1)\ez}g_j$.
Then $\vec h\equiv\{h_k\}_{k\in\zz}\in \bd^{s} (u)$.
For all $x\in\cx$ and  $k\in\zz$, by the H\"older inequality when $q\in(1,\,\fz)$ and
the inequality \eqref{e2.2} when $q\in(0,\,1]$,
we have
\begin{eqnarray*}
\sum_{j\ge k} [h_j ]^q &&\sim
\sum_{j\ge k} \lf [\sum_{i=j-N_3}^\fz 2^{-(i-j)\ez}g_i\r]^q
\ls
\sum_{j\ge k} \sum_{i=j-N_3}^\fz 2^{-(i-j)q\ez/2}\lf [g_i\r]^q
\ls
\sum_{i\ge k-N_3}\lf [g_i\r]^q,
\end{eqnarray*}
which together with $N_3\le0$ implies that
\begin{equation*}
\sum_{j\ge k}\bint_{B(x,\,2^{-k})} [h_j(y)]^q\,d\mu(y)
\ls \sum_{i\ge k}\bint_{B(x,\,2^{-k})} \lf [g_i(y)\r]^q\,d\mu(y).
\end{equation*}
Thus, (iii)  implies  (i).  This finishes the proof of Theorem \ref{t2.2}.
\end{proof}

\begin{rem}\label{r2.1}\rm
Comparing to Theorem \ref{t2.1}, notice that we require $N_1=0$ and  $N_3\le 0$ in Theorem \ref{t2.2}.
However, if $\cx$ has the doubling property, then
Theorem \ref{t2.2} still holds for all $N_1,\,N_2\in\zz_+$ and $N_3\in\zz$.
We omit the details.
\end{rem}

Finally, let $(\cx,\,d,\,\mu)$ be $\rn$ endowed with the Lebesgue measure and 
the Euclidean distance. 
The following Poincar\'e-type inequalities for  $\dot M^{s}_{p,\,q}(\rn)$ 
  play an important role in the following.

\begin{lem}\label{l2.1}
Let $s\in(0,\,1]$,  $p\in(1,\,\fz]$ and $q\in(0,\,\fz]$.
Then there exists a positive constant
$C$ such that for all $x\in\rn$, $k\in\zz$,
$u\in \dot M^{s}_{p,\,q}(B(x,\,2^{-k+2}))$
and $\vec g\in \bd^s(u)$,
\begin{eqnarray*}
   \inf_{c\in\rr}\dbint_{B(x,\, 2^{-k})}\lf|u(y)-c\r|\,dy
 \le C  2^{-ks} \sum_{j= k-3}^k\dbint_{B(x,\, 2^{-k+2})}
    g_j(y) \,dy.
\end{eqnarray*}
\end{lem}
\begin{proof}
Notice that for all $x\in\rn$ and $k\in\zz$,
\begin{eqnarray*}
   \inf_{c\in\rr}\bint_{B(x,\,2^{-k})}|u(y)-c|\,dy
&&\le   \bint_{B(x,\,2^{-k})}|u(y)-u_{B(x,\,2^{-k+2})\setminus B(x,\,2^{-k+1})}|\,dy.\\
&&\le \bint_{B(x,\,2^{-k})}\bint_{B(x,\,2^{-k+2})\setminus B(x,\,2^{-k+1})}|u(y)-u(z)|\,dy\,dz.\\
\end{eqnarray*}
Since for $y\in B(x,\,2^{-k})$ and $z\in B(x,\,2^{-k+2})\setminus B(x,\,2^{-k+1})$,
we have that $2^{-k}\le |y-z|<2^{-k+3}$, which implies that
$$|u(y)-u(z)|\ls 2^{-ks}\sum_{j=k-3}^{k-1}[g_j(y)+g_j(z)].$$
Thus,
\begin{eqnarray*}
\inf_{c\in\rr}\bint_{B(x,\,2^{-k})}|u(y)-c|\,dy
&&\ls 2^{-k s}\sum_{j=k-3}^{k-1}\bint_{B(x,\,2^{-k+2})}  g_j(y) \,dy,
\end{eqnarray*}
which completes the proof of Lemma \ref{l2.1}.
\end{proof}

The following inequality was  
given by Haj\l asz \cite[Theorem 8.7]{h03} when $s=1$, and  when $s\in (0,1)$,
it can be proved by a slight modification of the proof of
\cite[Theorem 8.7]{h03}.  

\begin{lem}\label{l2.2}
Let $s\in(0,\,1]$, $p\in(0,\,n/s)$
and $p_\ast=np/(n-sp)$. Then
there exists a positive constant
$C$ such that for all $x\in\rn$, $r\in(0,\,\fz)$, $u\in \dot M^{s,\,p}(B(x,\,2r))$ and $g\in \cd^s(u)$,
\begin{eqnarray*}
   \inf_{c\in\rr}\lf(\dbint_{B(x,\, r)}
    \lf|u(y)-c\r|^{p_\ast}\,dy\r)^{1/p_\ast}
 \le C  r^s\lf(\dbint_{B(x,\, 2r)}
    [g(y)]^{p}\,dy\r)^{1/p}.
\end{eqnarray*}
\end{lem}

\begin{lem}\label{l2.3}
Let $s\in(0,\,1]$ and $p\in(0,\,1]$.
Then for every pair of $\ez,\ez'\in(0,\,s)$ with $\ez<\ez'$,
there exists a positive constant
$C$ such that for all  $x\in\rn$, $k\in\zz$,
measurable functions $u$ and $\vec g\in \bd^s(u)$,
\begin{eqnarray}\label{e2.5}
&&\inf_{c\in\rr}\lf(\dbint_{B(x,\,2^{-k})}
    \lf|u(y)-c\r|^{np/(n-\ez p)}\,dy\r)^{(n-\ez p)/(np)}\\
 &&\quad\le C  2^{-k\ez'}\sum_{j\ge k-2}2^{-j(s-\ez')}\lf\{\dbint_{B(x,\,2^{-k+1})}
    [g_j(y)]^p\,dy\r\}^{1/p}.\nonumber
\end{eqnarray}
\end{lem}
\begin{proof}
For given $\ez,\,\ez'\in(0,\,s)$ with $\ez<\ez'$, and all $x\in\rn$ and $k\in\zz$,
without loss of generality, we may assume that the right-hand side of \eqref{e2.5} is finite.
For $\vec g\in\bd^s(u)$, taking $g\equiv \lf\{\sum_{j\ge k-2}2^{-j(s-\ez)p}(g_j)^p\r\}^{1/p}$,
 we have  that $g\in \cd^{\ez}(u)$ and $u\in \dot M^{\ez,\,p}(B(x,\,2^{-k+1}))$.
Indeed, for every pair of $y,\,z\in B(x,\,2^{-k+1})$,
there exists   $j\ge k-2$
such that $2^{-j-1}\le|y-z|< 2^{-j}$ and
hence
$$|u(y)-u(z)|\le |y-z|^s[g_j(y)+g_j(z)]\le |y-z|^{\ez} [g(y)+g(z)].  $$
Moreover, by \eqref{e2.2} with $q=p$, $\ez<\ez'$ and the H\"older inequality, we have
\begin{eqnarray}\label{e2.6}
&&\|g\|_{L^p(B(x,\,2^{-k+1}))}\\
&&\quad\le \lf\{\sum_{j\ge k-2}2^{-j(s-\ez)p}
\int_ {B(x,\,2^{-k+1})}[g_j(y)]^{p}\,dy\r\}^{1/p}\nonumber\\
&&\quad\ls\lf(\sum_{j\ge k-2} 2^{-j(\ez'-\ez)p/(1-p)}\r)^{(1-p)/p}\sum_{j\ge k-2}2^{-j(s-\ez')}\lf\{\int_{B(x,\,2^{-k+1})}
    [g_j(y)]^{p}\,dy\r\}^{1/p} \nonumber\\
&&\quad\ls  2^{-k(\ez'-\ez+n/p)} \sum_{j\ge k-2}2^{-j(s-\ez')}\lf\{\bint_{B(x,\,2^{-k+1})}
    [g_j(y)]^{p}\,dy\r\}^{1/p}.\nonumber
\end{eqnarray}
Thus, the above claims are true.

Then, applying Lemma \ref{l2.2}, we
obtain
\begin{eqnarray*}
&&\inf_{c\in\rr}\lf(\dbint_{B(x,\,2^{-k})}
    \lf|u(y)-c\r|^{np/(n-\ez p)}\,dy\r)^{(n-\ez p)/(np)}\\
&&\quad\ls 2^{-k\ez} \lf(\bint_{B(x,\,2^{-k+1}))} [g(y)]^{p}\,dy\r)^{1/p}\\
&&\quad\ls  2^{-k\ez' } \sum_{j\ge k-2}2^{-j(s-\ez')}\lf\{\bint_{B(x,\,2^{-k+1})}
    [g_j(y)]^{p}\,dy\r\}^{1/p},
\end{eqnarray*}
which together with \eqref{e2.6} gives \eqref{e2.5}.
This finishes the proof of Lemma \ref{l2.3}.
\end{proof}

\begin{rem}\label{r2.2}\rm
From Lemmas \ref{l2.1} through \ref{l2.3}, it is easy to see that for all $s\in(0,\,1]$,
$p\in(n/(n+s),\,\fz]$ and $q\in(0,\,\fz]$,
the elements of $\dot M^s_{p,\,q}(\rn)$ are actually  locally integrable.
\end{rem}

\section{Besov and Triebel-Lizorkin spaces on $\rn$}\label{s3}

In this section, with the aid of  grand Littlewood-Paley functions,
we characterize  full ranges of the classical Besov and Triebel-Lizorkin spaces on $\rn$ with $n\in\nn$
and establish their equivalence with the Haj\l asz-Besov and Haj\l asz-Triebel-Lizorkin spaces;
see Theorems \ref{t3.1} and \ref{t3.2}. 
In particular, Theorem \ref{t1.2} follows from (i) and (ii) of Theorem \ref{t3.2} with $p\in(0,\,\fz)$.

We first recall some notions and notation. 
In this section, we work on $\rn$ with $n\in\nn$. Recall that $\zz_+=\nn\cup\{0\}$. 
Let $\cs(\rn)$ be the {\it space of all Schwartz functions}, whose topology is determined by a
family of seminorms, $\{\|\cdot\|_{\cs_{k,\,\,m}(\rn)}\}_{k,\,m\in\zz_+}$, where for all
$k\in\zz_+$, $m\in(0,\,\fz)$ and $\vz\in\cs(\rn)$,
$$\|\vz\|_{\cs_{k,\,m}(\rn)}\equiv\sup_{\az\in\zz_+^n,\ |\az|\le
k}\sup_{x\in\rn}(1+|x|)^{m}|\partial^\az \vz(x)|.$$
Here, for any $\az\equiv(\az_1,\cdots,\az_n)\in\zz_+^n$,
$|\az|=\az_1+\cdots+\az_n$ and $\partial^\az\equiv
(\frac{\partial}{\partial x_1})^{\az_1}\cdots
(\frac{\partial}{\partial x_n})^{\az_n}$.
It is known that
$\cs(\rn)$ forms a locally convex topological vector space.
Denote by
$\cs'(\rn)$ the {\it dual space of $\cs(\rn)$} endowed with the
weak $\ast$-topology.
In what follows, for every $\vz\in\cs(\rn)$, $t>0$ and
$x\in\rn$, set $\vz_t(x)\equiv t^{-n}\vz(t^{-1}x)$.

Then the classical Besov and Triebel-Lizorkin spaces are defined as follows; see \cite{t83}.
\begin{defn}\label{d3.1}\rm
Let $s\in\rr$ and   $p,\,q\in(0,\,\fz]$. Let
$\vz\in\cs(\rn)$ satisfy that
\begin{equation}\label{e3.1}
\supp\widehat\vz\subset \{\xi\in\rn:\ 1/2\le|\xi|\le2\}\ {and}\
|\widehat\vz(\xi)|\ge {\rm constant}>0 \ {if}\ 3/5\le|\xi|\le5/3.
\end{equation}

(i) The {\it homogeneous Triebel-Lizorkin space} $\dot
F^s_{p,\,q}(\rn)$ is defined as the collection of all
$f\in\cs'(\rn)$ such that $\|f\|_{\dot F^s_{p,\,q}(\rn)}<\fz$,
where  when $p<\fz$,
\begin{equation*}
\|f\|_{ \dot F^s_{p,\,q}(\rn)}\equiv\lf\|\lf(\sum_{k\in\zz}2^{ksq}
  |\vz_{2^{-k}}\ast
f|^q\r)^{1/q}\r\|_{L^p(\rn)}
\end{equation*}
 with the usual modification made when $q=\fz$,
and when $p=\fz$,
\begin{equation*}
\|f\|_{ \dot F^s_{\fz,\,q}(\rn)}\equiv\sup_{x\in\rn}\sup_{\ell\in\zz}\lf(\bint_{B(x,\,2^{-\ell})}\sum_{k\ge\ell}2^{ksq}
 |\vz_{2^{-k}}\ast
f(y)|^q\,d\mu(y)\r)^{1/q}
\end{equation*}
with the usual modification made when $q=\fz$.

(ii) The {\it homogeneous  Besov space} $ \dot B^s_{p,\,q}(\rn)$ is
 defined as the collection of
all $f\in\cs'(\rn)$ such that  $ \|f\|_{ \dot B^s_{p,\,q}(\rn)}<\fz$, where
\begin{equation*}
\|f\|_{ \dot B^s_{p,\,q}(\rn)}\equiv\lf(\sum_{k\in\zz}2^{ksq}
\lf\|   \vz_{2^{-k}}\ast
f \r\|_{L^p(\rn)}^q\r)^{1/q}
\end{equation*}
with the usual modification made when $q=\fz$.
\end{defn}

\begin{rem}\label{r3.1}\rm
  Notice that if $\|f\|_{\dot
F^s_{p,\,q}(\rn)}=0$,
then it is easy to see that $f$ is a polynomial.
Denote by $\cp$ the {\it collection of all polynomials on $\rn$}.
So the quotient space $ \dot
F^s_{p,\,q}(\rn)/\cp $ is a quasi-Banach space.
By abuse of the notation,
 the space $ \dot
F^s_{p,\,q}(\rn)/\cp$ is always denoted by  $\dot
F^s_{p,\,q}(\rn)$,
and its element $[f]=f+\cp$ with $f\in \dot
F^s_{p,\,q}(\rn)$ simply by $f$.
A similar observation is also suitable to $\dot
B^s_{p,\,q}(\rn)$.
\end{rem}

Moreover, for each $N\in\zz_+$, denote by
$\cs_N(\rn)$ the {\it space of all functions $f\in \cs(\rn)$ satisfying
that $\int_\rn x^\az f(x)\,dx=0$ for all $\az\in\zz_+^n$ with
$|\az|\le N$}. For convenience, we also write
$\cs_{-1}(\rn)\equiv\cs(\rn)$.

For each $N\in\zz_+\cup\{-1\}$, $m\in(0,\,\fz)$ and $\ell\in\zz_+$,
we define the class  $\ca^\ell_{N,\,m}$ {\it of  test functions} by  
\begin{equation}\label{e3.2}
 \ca^\ell_{N,\,m}\equiv\{\phi\in\cs_N(\rn):\
\|\phi\|_{ \cs_{N+\ell+1,\,m}(\rn)}\le1\}.
\end{equation}
Then the grand Besov and Triebel-Lizorkin spaces are defined as follows.

\begin{defn}\label{d3.2}\rm
Let $s\in\rr$ and $q\in(0,\,\fz]$.
Let $\ca$ be a class of test functions as in \eqref{e3.2}.

(i)  The {\it homogeneous grand Triebel-Lizorkin space} $\ca\dot
F^s_{p,\,q}(\rn)$ is defined as the collection of all
$f\in\cs'(\rn)$ such that $\|f\|_{\ca\dot F^s_{p,\,q}(\rn)}<\fz$,
where  $\|f\|_{\ca\dot F^s_{p,\,q}(\rn)}$ is defined as in
$\|f\|_{ \dot F^s_{p,\,q}(\rn)}$
via replacing $|\vz_{2^{-k}}\ast u|$ by $\sup_{\phi\in\ca}|\phi_{2^{-k}}\ast u|$.

(ii) The {\it homogeneous grand Besov space} $\ca\dot B^s_{p,\,q}(\rn)$ is
 defined as the collection of
all $f\in\cs'(\rn)$ such that  $\|f\|_{\ca\dot B^s_{p,\,q}(\rn)}<\fz$,
where  $\|f\|_{\ca\dot B^s_{p,\,q}(\rn)}$ is defined as in
$\|f\|_{ \dot B^s_{p,\,q}(\rn)}$
via replacing $|\vz_{2^{-k}}\ast u|$ by $\sup_{\phi\in\ca}|\phi_{2^{-k}}\ast u|$.
\end{defn}

\begin{rem}\label{r3.2}\rm
For $\ca\equiv \ca^\ell_{N,\,m}$,  we also write $\ca\dot
F^s_{p,\,q}(\rn)$ as $\ca^\ell_{N,\,m}\dot F^s_{p,\,q}(\rn)$.
Moreover, if $N\in\zz_+$ and $\|f\|_{\ca\dot
F^s_{p,\,q}(\rn)}=0$, then it is easy to see that $f\in\cp_N$,
where $\cp_N$ is the {\it space
of polynomials with degree no more than $N$}.
So, similarly to Remark \ref{r3.1},
the quotient space $\ca\dot
 F^s_{p,\,q}(\rn)/\cp_N$ is always denoted by  $\ca \dot
F^s_{p,\,q}(\rn)$ and its element
 $[f]=f+\cp$   with $f\in\ca\dot
  F^s_{p,\,q}(\rn)$ simply by $f$.
A similar observation is also suitable to $\ca\dot
B^s_{p,\,q}(\rn)$.
\end{rem}

The main results  of this section read as follows.

\begin{thm}\label{t3.1}
Let $s\in\rr$ and $p,\,q\in(0,\,\fz]$.

(i) If
$J\equiv n/\min\{1,\,p,\,q\}$,
  $\ca=\ca^\ell_{N,\,m}$ with  $\ell\in\zz_+$,
$N+1>\max\{s,\,J-n-s\}$ and $m>\max\{J,\,n+N+1\}$,
then $\ca\dot F^s_{p,\,q}(\rn)=\dot F^s_{p,\,q}(\rn)$.

(ii) If $J\equiv n/\min\{1,\,p\}$,
  $\ca=\ca^\ell_{N,\,m}$ with  $\ell\in\zz_+$,
$N+1>\max\{s,\,J-n-s\}$ and $m>\max\{J,\,n+N+1\}$,
then  $\ca\dot B^s_{p,\,q}(\rn)=\dot B^s_{p,\,q}(\rn)$.
\end{thm}

\begin{thm}\label{t3.2} 
Let $\ca\equiv\ca^\ell_{0,\,m}$ with  $\ell\in\zz_+$  and $m>n+1$.

(i) If $s\in(0,\,1)$ and $p,\,q\in(n/(n+s),\,\fz]$, then
$\dot M^s_{p,\,q} (\rn)=\dot F^s_{p,\,q}(\rn)$.

(ii) If $s\in(0,\,1)$, $p \in(n/(n+s),\,\fz]$ and  $ q \in(0,\,\fz]$,
then
$\dot N^s_{p,\,q} (\rn)=\dot B^s_{p,\,q}(\rn)$.

(iii)  If $s\in(0,\,1]$ and $p,\,q \in(n/(n+s),\,\fz]$,
then $\dot M^s_{p,\,q} (\rn)=\ca\dot F^s_{p,\,q}(\rn)$.

(iv)
If $s\in(0,\,1]$, $p \in(n/(n+s),\,\fz]$ and  $ q \in(0,\,\fz]$, then
$\dot N^s_{p,\,q} (\rn)=\ca\dot B^s_{p,\,q}(\rn)$.
\end{thm}

\begin{rem}\label{r3.3}\rm
(i) Recall that
Theorem \ref{t3.1} for $\dot F^s_{p,\,q}(\rn)$ with $p<\fz$
was already given in \cite[Theorem 1.2]{kyz09}.
The proof of Theorem \ref{t3.1} for the full range of Besov and Triebel-Lizorkin spaces
is similar to that of \cite[Theorem 1.2]{kyz09}.
For the reader's convenience, we sketch it below.

(ii) For all $s\in(0,\,1)$ and $p\in(n/(n+s),\,\fz)$,
combining \cite[Corollary 1.3]{y03},
\cite[Corollary 1.2]{kyz09} and Proposition \ref{p2.1}, we already have
 $\dot M^{s}_{p,\,\fz}(\rn)=\dot M^{s,\,p}(\rn)=\dot F^s_{p,\,\fz}(\rn)$.

(iii) When $s=1$, as proved in \cite{h96,ks08},  $\dot M^{1,\,p} (\rn)=\dot W^{1,\,p}(\rn)$ for $p\in(1,\,\fz)$
 and $\dot M^{1,\,p}(\rn)=\dot H^{1,\,p} (\rn)$ for $p\in(n/(n+1),\,1]$,
 which together with Proposition \ref{p2.1} and \cite{t83} implies that
$\dot M^{1}_{p,\,\fz}(\rn)=\dot M^{1,\,p}(\rn)=\dot F^1_{p,\,2}(\rn)$
for all $p\in(n/(n+1),\,\fz)$.
Here $\dot W^{1,\,p}(\rn)$ with $p\in(1,\,\fz)$ denotes the homogeneous Sobolev space
and $\dot H^{1,\,p}(\rn)$ with $p\in(0,\,1]$   the homogeneous Hardy-Sobolev space.
\end{rem}

\begin{proof}[Proof of Theorem \ref{t3.1}]
Notice that $\|u\|_{\dot F^s_{p,\,q}(\rn)}\ls \|u\|_{\ca\dot F^s_{p,\,q}(\rn)}$
for all $u\in\cs'(\rn)$,
which implies that  $ \ca\dot F^s_{p,\,q}(\rn)\subset  \dot F^s_{p,\,q}(\rn)$.
Similarly, $ \ca\dot B^s_{p,\,q}(\rn)\subset  \dot B^s_{p,\,q}(\rn)$.
Conversely,  assume that $u\in \dot F^s_{p,\,q}(\rn)$ or $u\in \dot B^s_{p,\,q}(\rn)$.
Let $\psi\in\cs(\rn)$ satisfy the
same conditions as $\vz$  and
$\sum_{k\in\zz}\widehat{\vz}(2^{-k}\xi)\widehat\psi(2^{-k}\xi)=1$
for all $\xi\in\rn\setminus\{0\}$; see \cite[Lemma (6.9)]{FJW91} for the existence of $\psi$.
Then, by the Calder\'on reproducing formula,
for $f\in\cs'(\rn)$, there exist polynomials $P_u$ and
$\{P_i\}_{i\in\zz}$ depending on $f$ such that
\begin{equation}\label{e3.3}
u+P_u=\lim_{i\to-\fz}\lf\{\sum_{j=i}^\fz
 \vz_{2^{-j}}\ast\psi_{2^{-j}}\ast
u+P_i\r\},\end{equation}
where the series converges in
$\cs'(\rn)$; see, for example, \cite{Pe76,FJ85}.

Moreover, if $u\in \dot F^s_{p,\,q}(\rn)$ with $p\in(0,\,\fz)$,
then it is known that the degrees of the polynomials
$\{P_i\}_{i\in\zz}$ here are no more than $\lfloor s-n/p\rfloor$;
see \cite[pp.\,153-155]{FJ90} and \cite{FJ85}.
Furthermore, as shown in \cite[pp.\,153-155]{FJ90}, $u+P_u$ is
the canonical representative of $u$ in the sense that if $i=1,\,2$,
$\vz^{(i)},\,\psi^{(i)}$ satisfy \eqref{e3.1} and
$\sum_{k\in\zz}\widehat{
\vz^{(i)}}(2^{-k}\xi)\widehat{\psi^{(i)}}(2^{-k}\xi)=1$ for all
$\xi\in\rn\setminus\{0\}$, then
$P^{(1)}_{u}-P^{(2)}_u$ is a polynomial of degree no more than
$\lfloor s-n/p\rfloor$, where $P^{(i)}_u$ is as in \eqref{e3.3}
corresponding to $\vz^{(i)},\,\psi^{(i)}$ for $i=1,\,2$. So, in this
sense, we identify $u$ with $\wz u\equiv u+P_u$.

We point out that the above argument still holds when $u\in \dot B^s_{p,\,q}(\rn)$ or
$u\in \dot F^s_{p,\,q}(\rn)$ with the full range.
In fact, by \cite[pp.\,52-56]{Pe76}, if  $u\in \dot B^s_{p,\,\fz}(\rn)$, then the above arguments hold.
Moreover, by
 $\dot F^s_{\fz,\,q}(\rn)\subset \dot B^s_{\fz,\,\fz}(\rn)$ and
$ \dot B^s_{p,\,q}(\rn)\subset \dot B^s_{p,\,\fz}(\rn)$ for all possible $s,\,p$ and $q$,
 the above arguments hold for all $u\in \dot B^s_{p,\,q}(\rn)$ or
$u\in \dot F^s_{p,\,q}(\rn)$ with the full range.

Let $\widetilde\vz(x)\equiv{\vz(-x)}$ for all $x\in\rn$. Denote by
$\cq$ the collection of all  dyadic cubes on $\rn$. For every dyadic
cube $Q\equiv2^{-j}k+2^{-j}[0,\,1]^n$ with certain
$k\in\zz^n$,  set $x_Q\equiv2^{-j}k$, denote by
$\ell(Q)\equiv2^{-j}$ the side length of $Q$ and write
$\vz_Q(x)\equiv2^{j n/2}\vz(2^j x-k)=2^{-jn/2}\vz_{2^{-j}}(x-x_Q)$
for all $x\in\rn$.
Then for all $u\in \dot
F^s_{p,\,q}(\rn)$ or $u\in \dot
B^s_{p,\,q}(\rn)$, $\phi\in\cs_N(\rn)$ with  $N\ge \lfloor
s-n/p\rfloor$, $i\in\zz$ and $x\in\rn$, by \cite{FJ85,FJW91}, 
\cite[Lemma 2.8]{BH06}  and an argument as in the proof of \cite[Theorem 1.2]{kyz09}, we have
$$\wz u\ast\phi_{2^{-i}}(x)= \sum_{Q\in\cq}\langle u,\,
\wz\vz_Q \rangle \psi_Q\ast\phi_{2^{-i}}(x) =\sum_{Q\in\cq}t_Q\psi_Q
\ast\phi_{2^{-i}}(x),$$ where $t_Q=\langle u,\,\wz\vz_Q \rangle$.
Moreover, by the proof of \cite[Theorem 1.2]{kyz09} again, for all $R\in\cq$ with
$\ell(R)=2^{-i}$,
 we have
\begin{eqnarray*}
|\wz u\ast\phi_{2^{-i}}|\ls\sum_{\ell(R)=2^{-i}}\lf(\sum_{Q\in\cq}a_{RQ}t_Q\r)|R|^{-1/2}\chi_R,
\end{eqnarray*}
where
$$a_{RQ}\le \lf[\frac{\ell(R)}{\ell(Q)}\r]^s
\lf[1+\frac{|x_R-x_Q|}{\max\{\ell(R),\,\ell(Q)\}}\r]^{-J-\ez}
\min\lf\{\lf[\frac{\ell(R)}{\ell(Q)}\r]^{\frac{n+\ez}2},\,
\lf[\frac{\ell(Q)}{\ell(R)}\r]^{J+\frac{\ez-n}2} \r\}$$ for certain
$\ez>0$. If $J\equiv n/\min \{1,\,p,\,q\}$, then $\{a_{RQ}\}_{R,\,Q\in\cq}$ forms an almost diagonal
operator on $\dot f^s_{p,\,q}(\rn)$ and hence, is bounded
on  $\dot f^s_{p,\,q}(\rn)$, 
while if $J\equiv n/\min \{1,\,p\}$, then $\{a_{RQ}\}_{R,\,Q\in\cq}$ forms an almost diagonal
operator on $\dot f^s_{p,\,q}(\rn)$ and hence, is bounded on $\dot b^s_{p,\,q}(\rn)$;
see \cite[Theorem 3.3]{FJ90} and also
\cite[Theorem (6.20)]{FJW91}.
Here,  $\dot f^s_{p,\,q}(\rn)$ denotes the {\it set of all  sequences} $\{t_Q\}_{Q\in\cq}$ such that
\begin{equation*}
  \|\{t_Q\}_{Q\in\cq}\|_{\dot
f^s_{p,\,q}(\rn)}\equiv\lf\|\lf(\sum_{Q\in\cq}[|Q|^{-s/n-1/2}|t_Q|
\chi_Q]^q\r)^{1/q}\r\|_{L^p(\rn)}<\fz,
\end{equation*}
and
$\dot b^s_{p,\,q}(\rn)$ the {\it set of  all  sequences} $\{t_Q\}_{Q\in\cq}$ such that
\begin{equation*}
  \|\{t_Q\}_{Q\in\cq}\|_{\dot
b^s_{p,\,q}(\rn)}\equiv \lf\{\sum_{k\in\zz}\lf\|\sum_{Q\in\cq,\,\ell(Q)=2^{-k}}[|Q|^{-s/n-1/2}|t_Q|
\chi_Q]\r\|_{L^p(\rn)} ^q\r\}^{1/q}<\fz.
\end{equation*}

Moreover, by \cite[Theorem 2.2]{FJ90} or \cite[Theorem (6.16)]{FJW91},
$\|u\|_{\dot F^s_{p,\,q}(\rn)}\sim \|\{t_Q\}_{Q\in\cq}\|_{\dot
f^s_{p,\,q}(\rn)},$ which then implies that
\begin{eqnarray*}
\|\wz u\|_{\ca\dot F^s_{p,\,q}(\rn)} &&\ls \lf\|\lf\{\sum_{Q\in\cq}a_{RQ}t_Q \r\}_{R\in\cq}\r\|_{\dot f^s_{p,\,q}(\rn)}
\ls\lf\| \{t_Q \}_{Q\in\cq}\r\|_{\dot f^s_{p,\,q}(\rn)} \sim
\|u\|_{\dot F^s_{p,\,q}(\rn)}.
\end{eqnarray*}
This argument still holds with the spaces $ \dot F $ replaced by $\dot B $ due to the equivalence that
$\|u\|_{\dot B^s_{p,\,q}(\rn)}\sim \|\{t_Q\}_{Q\in\cq}\|_{\dot
b^s_{p,\,q}(\rn)}$ given by \cite[(1.11)]{FJ85}.
This finishes the proof of Theorem \ref{t3.1}.
\end{proof}

\begin{proof}[Proof of Theorem \ref{t3.2}.]
Observe that with the aid of Theorem \ref{t3.1}, (iii) and (iv) of Theorem \ref{t3.2} imply
(i) and (ii) of Theorem \ref{t3.2}.
So it suffices to prove (iii) and (iv) of Theorem \ref{t3.2}.

We first prove Theorem \ref{t3.2}(iii), namely, $\dot M^s_{p,\,q}(\rn)=\ca\dot F^s_{p,\,q}(\rn)$.
To prove $\dot M^s_{p,\,q}(\rn)\subset \ca\dot F^s_{p,\,q}(\rn)$, let
$u\in\dot M^s_{p,\,q}(\rn)$ and choose $\vec g\in\bd^s(u)$ such that
$\|\vec g\|_{L^p(\rn,\,\ell^q)}\ls \|u\|_{\dot M^s_{p,\,q}(\rn)}$.
Then for all $\phi\in\ca$, $x\in\rn$ and $k\in\zz$,
\begin{eqnarray*}
\lf|\phi_{2^{-k}}\ast u(x)\r|&&=\lf|\int_\rn\phi_{2^{-k}}(x-y)[u(y)-u_{B(x,\,2^{-k})}]\,dy\r|\\
&&\ls\sum_{j=0}^\fz2^{-2js}\bint_{B(x,\,2^{-k+j})}|u(y)-u_{B(x,\,2^{-k})}|\,dy.
\end{eqnarray*}
Since
$$\bint_{B(x,\,2^{-k+j})}|u(y)-u_{B(x,\,2^{-k})}|\,dy
\ls \sum_{i=0}^j\bint_{B(x,\,2^{-k+i})}|u(y)-u_{B(x,\,2^{-k+i})}|\,dy,$$
we then have
\begin{equation}\label{e3.4}
\lf|\phi_{2^{-k}}\ast u(x)\r|
 \ls\sum_{j=0}^\fz  2^{-2js}\bint_{B(x,\,2^{-k+j})}|u(y)-u_{B(x,\,2^{-k+j})}|\,dy.
\end{equation}

If $p,\,q\in(1,\,\fz]$, then by Lemma \ref{l2.1}, we have
\begin{eqnarray}\label{e3.5}
 \lf|\phi_{2^{-k}}\ast u(x)\r|
&&\ls \sum_{j=0}^\fz 2^{-2js} 2^{-ks+js} \sum_{i= k-j-3}^{k-j}\dbint_{B(x,\, 2^{-k+j+2})}
    g_i(y) \,dy\\
&&\ls 2^{-2ks}\sum_{j=-\fz}^{k } 2^{js} \bint_{B(x,\,2^{-j+2})}   g_j(z)\,dz
\ls 2^{-2ks}\sum_{j=-\fz}^{k } 2^{js}\cm(  g_{j})(x),\nonumber
\end{eqnarray}
where and in what follows, $\cm$ denotes the {\it Hardy-Littlewood maximal function}.

Thus, for $p,\,q\in(1,\,\fz)$,
  by the H\"older inequality and the Fefferman-Stein
vector-valued maximal inequality on $\cm$ (see \cite{fs71}),
we have
\begin{eqnarray}\label{e3.6}
\|u\|_{\ca\dot F^s_{p,\,q}(\rn)}
&&\ls \lf\|\lf(\sum_{k\in\zz}2^{-ksq}\lf[\sum^{k }_{j=-\fz}2^{js}
\cm( g_{j})\r]^q\r)^{1/q}\r\|_{L^p(\rn)}\\
&&\ls \lf\|\lf(\sum_{k\in\zz}2^{-ks}\sum^{k }_{j=-\fz}2^{js}
\lf[\cm(  g_{j})\r]^q\r)^{1/q}\r\|_{L^p(\rn)}\nonumber\\
&&\ls \lf\|\lf(\sum_{j\in\zz}
\lf[\cm( g_{j})\r]^q\r)^{1/q}\r\|_{L^p(\rn)}\ls\|\vec g\|_{L^p(\rn,\,\ell^q)}\ls \|u\|_{\dot M^s_{p,\,q}(\rn)}.\nonumber
\end{eqnarray}

If $p\in(n/(n+s),\,1]$ or $q\in(n/(n+s),\,1]$, by \eqref{e3.4} and Lemma \ref{l2.3}, 
choosing $\ez,\,\ez'\in(0,\,s)$ such that $\ez<\ez'$ and $n/(n+\ez')<\min\{p,\,q\}$,   for all $x\in\rn$,
\begin{eqnarray}\label{e3.7}
 \lf|\phi_{2^{-k}}\ast u(x)\r|
&&\ls \sum_{j=0}^\fz 2^{-2js} 2^{-(k-j)\ez'}
\sum_{i\ge k-j-2}2^{-i(s-\ez')}\\
&&\quad\quad\times \lf\{\dbint_{B(x,\,2^{-(k-j)+1})}
    [g_{i}(y)]^{n/(n+\ez)}\,dy\r\}^{(n+\ez)/n}\nonumber\\
&&\ls \sum_{i=-\fz}^{k-2} \sum_ {j\ge k-i-2} 2^{-2js} 2^{-(k-j)\ez'}
 2^{-i(s-\ez')}\nonumber\\
&&\quad\quad\times \lf\{\dbint_{B(x,\,2^{-(k-j)+1})}
    [g_{i}(y)]^{n/(n+\ez)}\,dy\r\}^{(n+\ez)/n}\nonumber\\
 &&\ls 2^{-2sk}\sum_{i=-\fz}^{k-2}2^{is}\lf[\cm([g_{i}]^{n/(n+\ez)})(x)\r]^{(n+\ez)/n}.\nonumber
\end{eqnarray}
Thus, for $p,\,q \in(n/(n+s),\,\fz)$ and
$\min\{p,\,q\}\in(n/(n+s),1]$,
by the H\"older inequality when $q\in(1,\,\fz)$, \eqref{e2.2} when $q\in(n/(n+s),\,1]$ and the
Fefferman-Stein vector-valued maximal inequality, similarly to \eqref{e3.6}, we obtain
\begin{eqnarray*}
\|u\|_{\ca\dot F^s_{p,\,q}(\rn)}
&&\ls \lf\|\lf(\sum_{k\in\zz}2^{-ksq}\lf[\sum^{k-2}_{j=-\fz}2^{js}
\lf[\cm([g_{j}]^{n/(n+\ez)})\r]^{(n+\ez)/n}\r]^q\r)^{1/q}\r\|_{L^p(\rn)}\\
&&\ls \lf\|\lf(\sum_{j\in\zz}
\lf [\cm([g_{j}]^{n/(n+\ez)})\r]^{(n+\ez)q/n}\r)^{1/q}\r\|_{L^p(\rn)} \ls\|\vec g\|_{L^p(\rn,\,\ell^q)}\ls \|u\|_{\dot M^s_{p,\,q}(\rn)}.\nonumber
\end{eqnarray*}

If $p\in(n/(n+s),\,\fz)$ and $q=\fz$,   by \eqref{e3.5}, \eqref{e3.7}, the
Fefferman-Stein vector-valued maximal inequality
and an argument similar to \eqref{e3.6}, we   have $\|u\|_{\ca\dot F^s_{p,\,q}(\rn)}
 \ls \|u\|_{\dot M^s_{p,\,q}(\rn)}.$

If $p=\fz$ and $q\in(1,\,\fz)$, then for all $x\in\rn$ and all $\ell\in\zz$,
by the H\"older inequality and \eqref{e3.5}, we have that
\begin{eqnarray*}
&&\ \bint_{B(x,\,2^{-\ell})}
\sum_{k\ge\ell}2^{ksq}\sup_{\phi\in\ca}|\phi_{2^{-k}}\ast u(z)|^q\,dz\\
&&\quad\ls  \bint_{B(x,\,2^{-\ell})}
\sum_{k\ge\ell}2^{-ksq}\lf[\sum_{j=-\fz}^{k } 2^{js} \bint_{B(z,\,2^{-j+2})}   g_j(y)\,dy\r]^q\,dz  \\
&&\quad\ls  \bint_{B(x,\,2^{-\ell})}
\sum_{k\ge\ell}2^{-ks}\sum_{j=-\fz}^{k} 2^{js} \lf[\bint_{B(z,\,2^{-j+2})}   g_j(y)\,dy\r]^q\,dz.
\end{eqnarray*}
We continue to estimate the last quantity by dividing $\sum_{j=-\fz}^{k} $ into
$\sum_{j=-\fz}^{\ell}$ and $\sum_{j=\ell+1}^{k} $ when $k>\ell$.
Notice that for all $z\in\rn$ and $j\in\zz$,  by the H\"older inequality, we obtain
\begin{eqnarray}\label{e3.8}
\lf[\bint_{B(z,\,2^{-j+2})}   g_j(y)\,dy\r]^q&&
\le \bint_{B(z,\,2^{-j+2})} \lf[  g_j(y)\r]^q\,dy
\le
 \|u\|^q_{\dot M^s_{\fz,\,q}(\rn)}.
\end{eqnarray}
From this, it follows that 
\begin{eqnarray*}
\bint_{B(x,\,2^{-\ell})}
\sum_{k\ge\ell}2^{-ks}\sum_{j=-\fz}^{\ell} 2^{js} \lf[\bint_{B(z,\,2^{-j+2})} g_j(y)\,dy\r]^q\,dz
 \ls  \|u\|^q_{\dot M^s_{\fz,\,q}(\rn)}.
\end{eqnarray*}
Moreover, since $B(z,\,2^{-j+2})\subset B(x,\,2^{-\ell+2})$ for all $j\ge \ell+1$ and all $z\in B(x,\,2^{-\ell})$,
 by the $L^q(\rn)$-boundedness of $\cm$, we  obtain  that
\begin{eqnarray*}
&&\bint_{B(x,\,2^{-\ell})}
\sum_{k>\ell}2^{-ks}\sum_{j=\ell+1}^{k} 2^{js} \lf[\bint_{B(z,\,2^{-j+2})}   g_j(y)\,dy\r]^q\,dz\\
&&\quad\ls\bint_{B(x,\,2^{-\ell})}
\sum_{k>\ell}2^{-ks}\sum_{j=\ell+1}^{k} 2^{js} \lf[\cm(  g_j\chi_{B(x,\,2^{-\ell+2})})(z)\r]^q\,dz  \\
&&\quad\ls  \sum_{j=\ell+1}^\fz\bint_{B(x,\,2^{-\ell })}
\lf[\cm(  g_j\chi_{B(x,\,2^{-\ell+2})})(z)\r]^q\,dz\\
&&\quad\ls  \sum_{j=\ell+1}^\fz\bint_{B(x,\,2^{-\ell+2})}
\lf[  g_j(y)\r]^q\,dy \ls  \|u\|^q_{\dot M^s_{\fz,\,q}(\rn)}.
\end{eqnarray*}
Thus, $\|u\|_{\ca\dot F^s_{\fz,\,q}(\rn)}\ls \|u\|_{\dot M^s_{\fz,\,q}(\rn)}.$

If $p=\fz$ and $q=\fz$, then the proof is similar but easier than the case $p=\fz$ and $q\in(0,\,\fz)$.
We omit the details.

If $p=\fz$ and $q\in(n/(n+s),\,1]$, then from  \eqref{e3.7} with
  $\ez\in(0,\,s)$ satisfying that $n/(n+\ez)<q$,
it follows that
\begin{eqnarray*}
&&\ \bint_{B(x,\,2^{-\ell})}
\sum_{k\ge\ell}2^{ksq}\sup_{\phi\in\ca}|\phi_{2^{-k}}\ast u(z)|^q\,dz\\
&&\quad\ls  \bint_{B(x,\,2^{-\ell})}
\sum_{k\ge\ell}2^{ksq}\sum_{j=0}^\fz 2^{-2jsq} 2^{-(k-j)\ez q}
\sum_{i\ge k-j-2}2^{-i(s-\ez)q}\nonumber\\
&&\quad\quad\quad\times  \lf\{\dbint_{B(x,\,2^{-(k-j)+1})}
    [g_{i}(y)]^{n/(n+\ez)}\,dy\r\}^{(n+\ez)q/n} \,dz  \\
&& \quad\ls  \bint_{B(x,\,2^{-\ell})}
\sum_{j=-\fz}^\fz  2^{-sq\max\{\ell,\,j\}} 2^{2jsq} 2^{-j\ez q}
\sum_{i\ge j-2}2^{-i(s-\ez)q}\nonumber\\
&&\quad\quad\quad\times  \lf\{\dbint_{B(x,\,2^{-j+1})}
    [g_{i}(y)]^{n/(n+\ez)}\,dy\r\}^{(n+\ez)q/n} \,dz.
\end{eqnarray*}
Notice that similarly to \eqref{e3.8}, for $i\ge j-1$ and $x\in\rn$, by the H\"older inequality and $q>n/(n+\ez)$,
we have
$$\lf\{\dbint_{B(x,\,2^{-j+1})}
    [g_{i}(y)]^{n/(n+\ez)}\,dy\r\}^{(n+\ez) /n}
\le \lf\{\dbint_{B(x,\,2^{-j+1})}
    [g_{i}(y)]^q\,dy\r\}^{1/q}\le\|u\|_{\dot M^s_{\fz,\,q}(\rn)},$$
which implies that
\begin{eqnarray*}
&&\bint_{B(x,\,2^{-\ell})}
\sum_{j=-\fz}^\ell 2^{-\ell sq}2^{2jsq} 2^{-j\ez q}\sum_{i\ge j-2}2^{-i(s-\ez)q}\\
&&\quad\quad\quad\times
\lf\{\dbint_{B(x,\,2^{-j+1})}
    [g_{i}(y)]^{n/(n+\ez)}\,dy\r\}^{(n+\ez)q/n} \,dz\\
 &&\quad\ls \|u\|^q_{\dot M^s_{\fz,\,q}(\rn)}
 \sum_{j=-\fz}^\ell 2^{-\ell sq}2^{jsq} \ls \|u\|^q_{\dot M^s_{\fz,\,q}(\rn)}.
\end{eqnarray*}
On the other hand, from $(n+\ez)q/n>1$ and the $L^{q(n+\ez)/n}(\rn)$-boundedness of $\cm$, it follows that
\begin{eqnarray*}
&&\bint_{B(x,\,2^{-\ell})}
\sum_{j=\ell+1}^\fz 2^{jsq} 2^{-j\ez q}
\sum_{i\ge j-2}2^{-i(s-\ez)q}\lf\{\dbint_{B(x,\,2^{-j+1})}
    [g_{i}(y)]^{n/(n+\ez)}\,dy\r\}^{(n+\ez)q/n} \,dz\\
&&\quad\ls\bint_{B(x,\,2^{-\ell})}
\sum_{i\ge \ell-1}\lf[\cm([g_{i}]^{n/(n+\ez)}\chi_{B(x,\,2^{-\ell})})(z)\r]^{(n+\ez)q/n} \,dz\\
&&\quad\ls \bint_{B(x,\,2^{-\ell})}
\sum_{i\ge \ell-1} [g_{i}(z)]^q \,dz\ls \|u\|^q_{\dot M^s_{\fz,\,q}(\rn)},
\end{eqnarray*}
which implies that $\|u\|_{\ca\dot F^s_{\fz,\,q}(\rn)}\ls \|u\|_{\dot M^s_{\fz,\,q}(\rn)}.$
We have completed the proof of that $\dot M^s_{p,\,q}(\rn) \subset\ca\dot F^s_{p,\,q}(\rn)$.

To  prove $\ca\dot F^s_{p,\,q}(\rn)\subset \dot M^s_{p,\,q}(\rn)$, let $u\in\ca\dot F^s_{p,\,q}(\rn)$. Since
$\ca\dot F^s_{p,\,q}(\rn)\subset \ca\dot F^s_{p,\,\fz}(\rn)=\dot M^{s,\,p}(\rn)\subset L^1_\loc(\rn)$
by Proposition \ref{p2.1} and Lemma \ref{l2.1} together with \cite[Corollary 2.1]{kyz09},
we know that $u\in L^1_\loc(\rn)$.
Fix $\vz\in\cs(\rn)$ with compact support
 and $\int_\rn \vz(x)\,dx=1$. Notice that
$\vz_{2^{-k}}\ast u(x)\to u(x)$ as $k\to\fz$
for almost all $x\in\rn$.
Then for almost all $x, y\in\rn$, letting $k_0\in\zz$ such that
$2^{-k_0-1}\le|x-y|< 2^{-k_0}$,
we have
\begin{eqnarray*}
        |u(x)-u(y)|&&\le|\vz_{2^{-k_0}}\ast u(x)-
\vz_{2^{-k_0}}\ast u(y)|\\
        &&\quad+\sum_{k\ge k_0}(|\vz_{2^{-k-1}}\ast u(x)-
\vz_{2^{-k}}\ast u(x)|+ |\vz_{2^{-k-1}}\ast u(y)- \vz_{2^{-k}}\ast
u(y)|).
\end{eqnarray*}
Write $\vz_{2^{-k_0}}\ast u(x)-\vz_{2^{-k_0}}\ast u(y)
=(\phi^{(x,\,y)})_{2^{-k_0}}\ast f(x)$ with
$\phi^{(x,\,y)}(z)\equiv\vz(z-2^{k_0}[x-y])-\vz(z)$ and
$\vz_{2^{-k-1}}\ast u(x)- \vz_{2^{-k}}\ast
u(x)=(\vz_{2^{-1}}-\vz)_{2^{-k}}\ast u(x)$. Notice that
$\vz_{2^{-1}}-\vz$ and $\phi^{(x,\,y)}$ are fixed constant multiples
of elements of $\ca$. For all $k\in\zz$ and $x\in\rn$, set
\begin{equation}\label{e3.9}
g_{k}(x)\equiv 2^{ks}
\sup_{\phi\in\ca }|\phi_{2^{-k}}\ast u(x)|.
\end{equation}
Then we have
\begin{eqnarray*}
        |u(x)-u(y)| \ls\sum_{k\ge k_0} 2^{-ks}[g_{k }(x)+g_{k }(y)],
\end{eqnarray*}
which means that $\vec g\equiv \{g_{k}\}_{k\in\zz}\in\wz \bd^{s,\,s,\,0}(u)$.

Thus, if $p\in(n/(n+s),\,\fz)$, then  Theorem \ref{t2.1} implies that
\begin{eqnarray*}
\|u\|_{\dot M^s_{p,\,q}(\rn)}&&\ls \|\vec g\|_{L^p(\rn;\,\ell^q)}\ls \lf\|\lf(\sum_{j\in\zz} 2^{j sq}
\sup_{\phi\in\ca}|\phi_{2^{-j}}\ast u|^q\r)^{1/q}\r\|_{L^p(\rn)}\ls
\|u\|_{\ca\dot F^s_{p,\,q}(\rn)}.
\end{eqnarray*}

If $p=\fz$ and $q\in(n/(n+s),\,\fz)$, 
then by Theorem \ref{t2.2}, we obtain
\begin{eqnarray*}
 \bint_{B(x,\,2^{-\ell})} \sum_{k\ge\ell} [g_k(y)]^q\,dy
&&\ls \bint_{B(x,\,2^{-\ell})} \sum_{k\ge\ell}2^{ksq}
 \sup_{\phi\in\ca}|\phi_{2^{-k}}\ast u(y)|^q\,dy
 \ls
\|u\|_{\ca\dot F^s_{\fz,\,q}(\rn)}, 
\end{eqnarray*}
for all $x\in\rn$ and $\ell\in\zz$. 
Thus, $\|u\|_{\dot M^s_{\fz,\,q}(\rn)}\ls\|u\|_{\ca\dot F^s_{\fz,\,q}(\rn)}.$

If $p=\fz$ and $q=\fz$, the proof is similar and easier. We omit the details.
This finishes the proof of Theorem \ref{t3.2}(iii).

Now, we prove Theorem \ref{t3.2}(iv), namely, $\dot N^s_{p,\,q}(\rn)= \ca \dot B^s_{p,\,q}(\rn)$.
To prove $\dot N^s_{p,\,q}(\rn)\subset \ca\dot B^s_{p,\,q}(\rn)$,
let $\ez\in(0,\,s)$ such that $n/(n+\ez)<p$ and notice that \eqref{e3.7} still holds here.
Then for all $u\in\dot N^s_{p,\,q}(\rn)$ and $\vec g\in\bd^s(u)$,
by \eqref{e3.7}, we have
\begin{eqnarray*}
\|u\|_{\ca\dot B^s_{p,\,q}(\rn)}&&\ls
\lf(\sum_{k\in\zz}2^{-ksq}\lf\|\sum_{j=-\fz}^{k-2} 2^{js}[\cm([  g_{j}]^{n/(n+\ez)})]^{(n+\ez)/n}\r\|_{L^p(\rn)}^q\r)^{1/q}.
\end{eqnarray*}

Now we consider two cases.
If $p\in(n/(n+\ez),\,1]$, by \eqref{e2.2} with $q$ there replaced by $p$, we further obtain
\begin{eqnarray*}
\|u\|_{\ca\dot B^s_{p,\,q}(\rn)}
&&\ls
\lf\{\sum_{k\in\zz}2^{-ksq }\lf(\sum_{j=-\fz}^{k-2} 2^{jsp }\lf\|[\cm([  g_{j}]^{n/(n+\ez)})]^{(n+\ez)/n}\r\|^p_{L^p(\rn)}\r)^{q/p}\r\}^{1/q}.
\end{eqnarray*}
From this, the H\"older inequality when $q>p$ and
 \eqref{e2.2} with $q$ there replaced by $q/p$ when $q\le p$,
and the $L^{p(n+\ez)/n}(\rn)$-boundedness of $\cm$, it follows that
\begin{eqnarray*}
\|u\|_{\ca\dot B^s_{p,\,q}(\rn)}
&&\ls
\lf(\sum_{k\in\zz}2^{-ksq/2}\sum_{j=-\fz}^{k-2} 2^{jsq/2}\lf\|  g_{j}\r\|_{L^p(\rn)}^q\r)^{1/q}\\
&&\sim  \lf(\sum_{j\in\zz}
 \lf\| g_{j}\r\|_{L^p(\rn)}^q\r)^{1/q}\sim\|\vec g\|_{\ell^q(L^p(\rn))}\ls \|u\|_{\dot N^s_{p,\,q}(\rn)}.
\end{eqnarray*}

If $p\in(1,\,\fz]$, then by the Minkowski inequality, we have
\begin{eqnarray*}
\|u\|_{\ca\dot B^s_{p,\,q}(\rn)}&&\ls
\lf\{\sum_{k\in\zz}2^{-ksq}\lf(\sum_{j=-\fz}^{k } 2^{js}\lf\|[\cm([  g_{j}]^{n/(n+\ez)})]^{(n+\ez)/n}\r\|_{L^p(\rn)}\r)^q\r\}^{1/q},
\end{eqnarray*}
which together with the H\"older inequality or \eqref{e2.2} when $q\in(0,\,1]$,
and the $L^{p(n+\ez)/n}(\rn)$-boundedness of $\cm$ also yields that
\begin{eqnarray*}
\|u\|_{\ca\dot B^s_{p,\,q}(\rn)}
&&\ls
\lf(\sum_{k\in\zz}2^{-ksq/2}\sum_{j=-\fz}^{k-2} 2^{jsq/2}\lf\|  g_{j}\r\|_{L^p(\rn)}^q\r)^{1/q}\ls \|u\|_{\dot N^s_{p,\,q}(\rn)}.
\end{eqnarray*}
Thus, $\dot N^s_{p,\,q}(\rn)\subset \ca\dot B^s_{p,\,q}(\rn)$.

Conversely, to show $\ca\dot B^s_{p,\,q}(\rn)\subset \dot N^s_{p,\,q}(\rn)$,
let $u\in \ca\dot B^s_{p,\,q}(\rn)$.
Then we claim that $u\in L^1_\loc(\rn)$.
Assume that this claim holds for the moment.
Taking $\vec g\equiv\{g_k\}_{k\in\zz}$ with $g_k$ as in \eqref{e3.9} and by an argument similar to
the proof of $\ca\dot F^s_{p,\,q}(\rn)\subset \dot M^s_{p,\,q}(\rn)$,
we know that $\vec g\in\wz \bd^{s,\,s,\,0}(u)$.
By Theorem \ref{t2.1}, we have
\begin{eqnarray*}
\|u\|_{\dot N^s_{p,\,q}(\rn)}&&\ls\lf(\sum_{j\in\zz}\|g_j\|^q_{L^p(\rn)}\r)^{1/q}\\
&&\ls \lf(\sum_{j\in\zz} 2^{j sq}\lf\|
\sup_{\phi\in\ca}|\phi_{2^{-j}}\ast u|\r\|_{L^p(\rn)}^q\r)^{1/q}\ls
\|u\|_{\ca\dot B^s_{p,\,q}(\rn)},
\end{eqnarray*}
which  implies that $\ca\dot B^s_{p,\,q}(\rn)\subset \dot N^s_{p,\,q}(\rn)$.

Finally, we prove the above claim that  $u\in L^1_\loc(\rn)$.
If $p=\fz$, since 
$$\ca\dot B^s_{\fz,\,q}(\rn)\subset \ca\dot B^s_{\fz,\,\fz}(\rn)
=\ca\dot F^s_{\fz,\,\fz}(\rn)=\dot M^{s,\,\fz}(\rn)\subset L^1_\loc(\rn)$$ by   \cite[Corollary 1.2]{kyz09},
then $u\in  L^1_\loc(\rn)$.
For  $p\in(n/(n+s),\,\fz)$, let  $\vz\in\cs(\rn)$ satisfy  $\int_\rn\vz(z)\,dz=1$.
Then
 $ \vz_{2^{-k}}\ast u\to u$ in $\cs'(\rn)$ and hence
 $$u=\vz\ast u+\sum_{k=0}^\fz(\vz_{2^{-k-1}}\ast u-\vz_{2^{-k}}\ast u)$$ in $\cs'(\rn)$.
Observe that for all $x\in\rn$ and $k\in\zz_+$,
\begin{eqnarray*}
 |\vz_{2^{-k-1}}\ast u(x)-\vz_{2^{-k}}\ast u(x)|&&\ls
\sup_{\phi\in\ca}|\phi_{2^{-k}}\ast u(x)|.
\end{eqnarray*}
If  $p\in[1,\,\fz)$, then
 $$\sum_{k=0}^\fz\|\vz_{2^{-k-1}}\ast u-\vz_{2^{-k}}\ast u \|_{L^p(\rn)}
\ls \sum_{k=0}^\fz 2^{-ks}\|u\|_{\ca \dot B^s_{p,\,q}(\rn)}\ls\|u\|_{\ca \dot B^s_{p,\,q}(\rn)},$$
which implies that
$\sum_{k=0}^\fz(\vz_{2^{-k-1}}\ast u-\vz_{2^{-k}}\ast u)$ converges in $L^p(\rn)$.
Observing that $\vz\ast u$ is a continuous function, we know that
$\vz\ast u+\sum_{k=0}^\fz(\vz_{2^{-k-1}}\ast u-\vz_{2^{-k}}\ast u)\in L^1_\loc(\rn)$,
which implies that
$u$ is an element of $\cs'(\rn)$
induced by a function in $L^1_\loc(\rn)$. In this sense, we say that $u\in L^1_\loc(\rn)$.
For $p\in(n/(n+s),\,1)$, it is easy to see that
for all $\phi\in\ca$, $k\in\zz$, $x\in\rn$ and $y\in B(x,\,2^{-k})$,
the function $\wz\phi(z)\equiv\phi(z+2^k(x-y))$ for all $z\in\rn$
is a  constant multiple of an element of $\ca$ with the constant independent of $x, y$ and $k$.
Notice that $\phi_{2^{-k}}\ast u(x)=\wz\phi_{2^{-k}}\ast u(y)$.
Then for all $k\in\zz$ and $x\in\rn$,
\begin{eqnarray*}
 \sup_{\phi\in\ca}|\phi_{2^{-k}}\ast u(x)|
&&=\lf(\bint_{B(x,\,2^{-k})}\sup_{\phi\in\ca}|\phi_{2^{-k}}\ast u(x)|^p\,dy\r)^{1/p}\\
&&\ls \lf(\bint_{B(x,\,2^{-k})}\sup_{\phi\in\ca}|\phi_{2^{-k}}\ast u(y)|^p\,dy\r)^{1/p}\ls
2^{kn/p}\lf\|\sup_{\phi\in\ca}|\phi_{2^{-k}}\ast u|\r\|_{L^p(\rn)}
\end{eqnarray*}
and hence
\begin{eqnarray*}
 \lf\|\sup_{\phi\in\ca}|\phi_{2^{-k}}\ast u|\r\|_{L^1(\rn)}&&\ls2^{k(1-p)n/p}
 \lf\|\sup_{\phi\in\ca}|\phi_{2^{-k}}\ast u|^p\r\|_{L^1(\rn)}
\lf\|\sup_{\phi\in\ca}|\phi_{2^{-k}}\ast u|\r\|^{1-p}_{L^p(\rn)}\\
&&\ls2^{kn(1/p-1)}
\lf\|\sup_{\phi\in\ca}|\phi_{2^{-k}}\ast u|\r\|_{L^p(\rn)},
\end{eqnarray*}
which together with $p>n/(n+s)$ implies that
$$\sum_{k=0}^\fz\|\vz_{2^{-k-1}}\ast u-\vz_{2^{-k}}\ast u \|_{L^1(\rn)}
\ls \sum_{k=0}^\fz 2^{-k(n+s-n /p)}\|u\|_{\ca \dot B^s_{p,\,q}(\rn)}\ls\|u\|_{\ca \dot B^s_{p,\,q}(\rn)}.$$
From this and an argument similar to the case $p\in[1,\,\fz)$, it follows that
$u\in L^1_\loc(\rn)$. This shows the above claim and finishes the proof of Theorem \ref{t3.2}(iv) and hence
of  Theorem \ref{t3.2}.
\end{proof}

\section{Besov and Triebel-Lizorkin spaces on RD-spaces}\label{s4}

Let $(\cx,\,d,\,\mu)$ be an {\it RD-space} throughout the whole section.
We extend Theorems \ref{t3.1} and \ref{t3.2} to
the Besov and Triebel-Lizorkin spaces on $\cx$; see Theorem \ref{t4.1}.
We also establish an equivalence of $\dot M^s_{p,\,p}(\cx) $
and the Besov space $\dot \cb^s_p(\cx) $  considered by Bourdon and Pajot \cite{bp03};
see Proposition \ref{p4.1}.

We begin with the definition of  the homogeneous (grand) Besov and Triebel-Lizorkin spaces
on RD-spaces. To this end, we first recall the spaces of test functions on RD-spaces; see \cite{hmy08}.
For our convenience, in what follows, 
for any $x,$ $y\in\cx$ and $r>0$, we always set $V(x,
y)\equiv\mu(B(x, d(x,y)))$ and $V_r(x)\equiv\mu(B(x, r))$. It is
easy to see that $V(x,y)\sim V(y,x)$ for all $x,\,y\in\cx$.
Moreover, if $\mu(\cx)<\fz$, then $\diam \cx<\fz$ and hence without loss of generality,
we may always assume that $\diam\cx=2^{-k_0}$ for some $k_0\in\zz$.

\begin{defn}\label{d4.1}\rm
Let $x_1\in\cx$, $r\in(0, \fz)$, $\bz\in(0, 1]$ and $\gz\in(0,
\fz)$. A function $\vz$ on $\cx$ is said to be in the {\it space
$\cg(x_1, r, \bz, \gz)$} if there exists a nonnegative constant
$C$ such that

(i) $|\vz(x)|\le C\frac{1}{V_r(x_1)+V(x_1, x)}
\lf(\frac{r}{r+d(x_1, x)}\r)^\gz$ for all $x\in\cx$;

(ii) $|\vz(x)-\vz(y)| \le C\lf(\frac{d(x, y)}{r+d(x_1, x)}\r)^\bz
\frac{1}{V_r(x_1)+V(x_1, x)} \lf(\frac{r}{r+d(x_1, x)}\r)^\gz$
for all $x$, $y\in\cx$ satisfying that $d(x, y)\le(r+d(x_1, x))/2$.

Moreover, for any $\vz\in\cg(x_1, r, \bz,\gz)$, its {\it norm} is defined
by $$\|\vz\|_{\cg(x_1,\, r,\, \bz,\,\gz)} \equiv\inf\{C:\  (i) \mbox{
and } (ii) \mbox{ hold}\}.$$
\end{defn}

Throughout this section, we fix $x_1\in \cx $ and let $\cg(\bz,\gz)
\equiv\cg(x_1,1,\bz,\gz).$ Then
$\cg(\bz, \gz)$ is a Banach space.
We also let
$\ocg(\bz, \gz)\equiv\lf\{f\in\cg(\bz, \gz):\,\int_\cx f(x)\,d\mu(x)=0\r\}.$
Denote by $(\cg(\bz, \gz))'$ and $(\ocg(\bz, \gz))'$
the {\it dual spaces} of $\cg(\bz, \gz)$ and $\ocg(\bz, \gz)$, respectively.
Obviously, $(\ocg(\bz, \gz))'=(\cg(\bz, \gz))'/\cc$.

Let $\ez\in(0, 1]$ and 
$\bz$, $\gz\in(0, \ez)$. 
Define  $\cg_0^\ez(\bz, \gz)$ as the
{\it completion of the set $\cg(\ez, \ez)$ in the space $\cg(\bz, \gz)$},
and for $\vz\in\cg_0^\ez(\bz, \gz)$, define
$\|\vz\|_{\cg_0^\ez(\bz, \gz)}\equiv\|\vz\|_{\cg(\bz, \gz)}$. 
Then, it is easy to see that $\cg_0^\ez(\bz, \gz)$ is a Banach space. 
Similarly, we define  $\ocg_0^\ez(\bz,\gz)$.  Let
$(\cg_0^\ez(\bz, \gz))'$ and $(\ocg_0^\ez(\bz, \gz))'$ be the {\it dual
spaces} of $\cg_0^\ez(\bz, \gz)$ and $\ocg_0^\ez(\bz, \gz)$,
respectively. Obviously,
$(\ocg^\ez_0(\bz,\,\gz))'=(\cg^\ez_0(\bz,\,\gz))'/\cc$.

Now we recall the notion of approximations of the identity on
RD-spaces, which were first introduced in \cite{hmy08}.

\begin{defn}\label{d4.2}\rm
Let $\ez_1\in(0,\, 1]$ and assume that $\mu(\cx)=\fz$. A sequence $\{S_k\}_{k\in\zz}$ of bounded
linear integral operators on $L^2(\cx)$ is called an {\it approximation
of the identity of order $\ez_1$ with bounded support} (for short, $\ez_1$-$\aoti$ {\it with bounded support}),
if there exist positive constants $C_3$ and $C_4$ such that
for all $k\in\zz$ and all $x$, $x'$, $y$ and $y'\in\cx$, $S_k(x,
y)$, the integral kernel of $S_k$ is a measurable function from
$\cx\times\cx$ into $\cc$ satisfying
\begin{enumerate}
\item[(i)] $S_k(x, y)=0$ if $d(x, y)>C_42^{-k}$
and $|S_k(x, y)|\le C_3\frac{1}{V_{2^{-k}}(x)+V_{2^{-k}}(y)};$
\vspace{-0.2cm}
\item[(ii)] $|S_k(x, y)-S_k(x', y)|
\le C_3 2^{k \ez_1}d(x, x')^{\ez_1}
\frac{1}{V_{2^{-k}}(x)+V_{2^{-k}}(y)}$ for $d(x, x')\le\max\{C_4,
1\}2^{1-k}$; \vspace{-0.5cm}
\item[(iii)] Property (ii) holds with $x$ and $y$ interchanged;
\vspace{-0.2cm}
\item[(iv)] $|[S_k(x, y)-S_k(x, y')]-[S_k(x', y)-S_k(x', y')]|
\le C_3 2^{2k\ez_1}\frac{[d(x, x')]^{\ez_1}[d(y, y')]^{\ez_1}}
{V_{2^{-k}}(x)+V_{2^{-k}}(y)}$ for $d(x, x')\le\max\{C_4,
1\}2^{1-k}$ and $d(y, y')\le\max\{C_4, 1\}2^{1-k}$; \vspace{-0.2cm}
\item[(v)] $\int_\cx S_k(x, z)\, d\mu(z)
=1=\int_\cx S_k(z, y)\, d\mu(z)$.
\end{enumerate}
\end{defn}

\begin{rem}\label{r4.1}\rm
It was proved in \cite[Theorem~2.6]{hmy08} that there always exists a
$1$-$\aoti$ with bounded support on RD-spaces.
\end{rem}

Recall the notion of homogeneous Triebel-Lizorkin
spaces  in \cite{hmy08} as follows.

\begin{defn}\label{d4.3}\rm
Let $\ez\in (0,1)$,
 $s\in(0,\,\ez)$ and $p\in(n/(n+\ez),\,\fz]$.
Let $\bz,\ \gz\in(0,\,\ez)$ such that
$
\bz\in(s,\,\ez)\ {\rm and}\ \gz\in(\max\{s-\kz/p,\ n/p-n,\,0\},\,\ez).
$
Assume that $\mu(\cx)=\fz$ and  $\{S_k\}_{k\in\zz}$ is an  $\ez$-$\aoti$ with bounded support
as in Definition \ref{d4.2}.
For $k\in\zz$, set $D_k\equiv S_k-S_{k-1}$.

(i)  Let  $q\in(n/(n+\ez),\,\fz].$ The  {\it homogeneous
Triebel-Lizorkin space} $\dot F^s_{p,\,q}(\cx)$ is defined to be the set of all
$f\in(\ocg^\ez_0(\bz,\gz))'$ such that$ \|f\|_{\dot F^s_{p,\,q}(\cx)}<\fz$,
where when $p\in(n/(n+\ez),\,\fz)$,
\begin{equation}\label{e4.1}
\|f\|_{\dot F^s_{p,\,q}(\cx)}\equiv\lf\|\lf\{\sum^\fz_{k=-\fz}
2^{ksq}|D_k(f)|^q\r\}^{1/q}\r\|_{L^p(\cx)}
\end{equation}
 with the usual
modification made when $q=\fz$,
while when $p=\fz$,
\begin{equation}\label{e4.2}
\|f\|_{\dot F^s_{\fz,\,q}(\cx)}\equiv\sup_{x\in\cx}\sup_{\ell\in\zz}\lf(\bint_{B(x,\,2^{-\ell})}\sum_{k\ge\ell}2^{ksq}
 |D_k(f)(y)|^q\,d\mu(y)\r)^{1/q}
\end{equation}
with the usual modification made when $q=\fz$.

(ii)  Let  $q\in(0,\,\fz].$ The {\it homogeneous
Besov space} $\dot B^s_{p,\,q}(\cx)$ is defined to be the set of all
$f\in(\ocg^\ez_0(\bz,\gz))'$ such that
\begin{equation}\label{e4.3}
\|f\|_{\dot B^s_{p,\,q}(\cx)}\equiv\lf\{\sum^\fz_{k=-\fz}
2^{ksq} \lf\|D_k(f) \r\|^q_{L^p(\cx)}\r\}^{1/q}<\fz
\end{equation}
 with the usual
modification made when $q=\fz$.
\end{defn}

\begin{rem}\label{r4.2}\rm
(i)  As shown in \cite{yz09}, the definition of $\dot F^s_{p,\,q}(\cx)$ is independent
of the choices of $\ez$, $\bz$, $\gz$ and the approximation of the
identity as in Definition \ref{d4.2}.

(ii) By $(\ocg^\ez_0(\bz,\,\gz))'=(\cg^\ez_0(\bz,\,\gz))'/\cc$, if we
replace  $(\ocg^\ez_0(\bz,\,\gz))'$ with   $(\cg^\ez_0(\bz,\,\gz))'/\cc$ or similarly with 
 $(\cg^\ez_0(\bz,\,\gz))'$ in Definition \ref{d4.3}, 
then we obtain new Besov and Triebel-Lizorkin spaces, 
which modulo constants are equivalent to the original Besov and 
Triebel-Lizorkin spaces respectively. So we can replace $(\ocg^\ez_0(\bz,\,\gz))'$
with $(\cg^\ez_0(\bz,\,\gz))'/\cc$ or $(\cg^\ez_0(\bz,\,\gz))'$ in
the Definition \ref{d4.3} if need be, in what follows.
\end{rem}

To define grand Besov and Triebel-Lizorkin spaces,
 we introduce the class of test functions.
Motivated by \cite{kyz09},
when $\mu(\cx)=\fz$,  for all $x\in\cx$ and $k\in\zz$,  let
\begin{equation}\label{e4.4} \ca_k(x)\equiv \{\phi\in\ocg(1,\,2):\
\|\phi\|_{\ocg(x,\,2^{-k},\,1,\,2)}\le1\};
\end{equation}
when $\mu(\cx)=2^{-k_0}$, for all $x\in\cx$ and $k\ge k_0$,  let
$\ca_k(x)$ be as in \eqref{e4.4}, and for $k<k_0$, let $\ca_k(x)\equiv\{0\}$.
Set $\ca\equiv\{\ca_k(x)\}_{x\in\cx,\,k\in\zz}$.
Moreover,  we also introduce the class of test functions with bounded support.
For all $x\in\cx$ and $k\in\zz$,
let
\begin{equation}\label{e4.5}
\wz\ca_k(x)\equiv \{\phi\in \ca_k(x):\ \supp\phi\subset B(x,\,2^{-k})\}.
\end{equation}
Set $\wz\ca\equiv\{\wz\ca_k(x)\}_{x\in\cx,\,k\in\zz}$.

\begin{defn}\label{d4.4}\rm
Let $s\in(0,\,1]$, $p,\,q\in(0,\,\fz]$ and $\ca$ be as above. 

(i)   The {\it homogeneous
grand Triebel-Lizorkin space} $\ca\dot F^s_{p,\,q}(\cx)$
is defined to be the set of all
$f\in(\cg(1,\,2))'$ that satisfy
$\|f\|_{\ca\dot F^s_{p,\,q}(\cx)}<\fz$,
where  $\|f\|_{\ca\dot F^s_{p,\,q}(\cx)}$ is defined as  $\|f\|_{\dot F^s_{p,\,q}(\cx)}$ via replacing
$|D_k(f)|$ in \eqref{e4.1} and \eqref{e4.2} by $\sup_{\vz\in\ca_k}|\langle f,\,\vz\rangle|$.

(ii)  The {\it homogeneous
grand Besov space} $\ca\dot B^s_{p,\,q}(\cx)$
is defined to be the set of all
$f\in(\cg(1,\,2))'$ that satisfy
$\|f\|_{\ca\dot F^s_{p,\,q}(\cx)}<\fz$,
where  $\|f\|_{\ca\dot B^s_{p,\,q}(\cx)}$ is defined as  $\|f\|_{\dot B^s_{p,\,q}(\cx)}$ via replacing
$|D_k(f)|$ in \eqref{e4.3} by $\sup_{\vz\in\ca_k}|\langle f,\,\vz\rangle|$.

Define the {\it spaces} $\wz\ca\dot F^s_{p,\,q}(\cx)$ and $\wz\ca\dot B^s_{p,\,q}(\cx)$ as
 $\ca\dot F^s_{p,\,q}(\cx)$ and $ \ca\dot B^s_{p,\,q}(\cx)$ via replacing $\ca$ by  $\wz\ca$ as in  \eqref{e4.5}.
\end{defn}

The main result of this section is as follows. 

\begin{thm}\label{t4.1}
(i) Assume that $\mu(\cx)=\fz$. If $s\in(0,\,1)$ and
$p,\,q\in(n/(n+s),\,\fz]$, then  $ \dot F^s_{p,\,q}(\cx)=\dot M^s_{p,\,q}(\cx)$.

(ii) Assume that $\mu(\cx)=\fz$. If $s\in(0,\,1)$, $p \in(n/(n+s),\,\fz]$ and $q\in(0,\,\fz]$,
then $ \dot B^s_{p,\,q}(\cx)= \dot N^s_{p,\,q}(\cx)$.

(iii) If $s\in(0,\,1]$ and $p,\,q\in(n/(n+s),\,\fz]$, then
 $\ca\dot F^s_{p,\,q}(\cx)=\wz\ca\dot F^s_{p,\,q}(\cx) =\dot M^s_{p,\,q}(\cx)$.

(iv) If $s\in(0,\,1]$, $p \in(n/(n+s),\,\fz]$ and $q\in(0,\,\fz]$,
then $\ca\dot B^s_{p,\,q}(\cx)=\wz\ca\dot B^s_{p,\,q}(\cx) = \dot N^s_{p,\,q}(\cx)$.

\end{thm}

\begin{proof}
The proof of Theorem \ref{t4.1} uses the ideas from
Theorems \ref{t3.1} and \ref{t3.2},
and also some from \cite[Theorems 1.4 and 5.1]{kyz09}.
We only point out that all the tools to prove Theorem \ref{t4.1} are available.
The details are omitted.

Assume that $\mu(\cx)=\fz$.
Then the result $\ca\dot F^s_{p,\,q}(\cx)= \dot F^s_{p,\,q}(\cx)$ for $p<\fz$ is given in \cite[Theorem 1.4]{kyz09},
whose proof used the discrete Calder\'on
reproducing formula established in \cite{hmy08}.
The proofs of
$\ca\dot F^s_{\fz,\,q}(\cx)= \dot F^s_{\fz,\,q}(\cx)$ and
$\ca\dot B^s_{p,\,q}(\cx)= \dot B^s_{p,\,q}(\cx)$
can be done by using the discrete Calder\'on
reproducing formula and an argument
similar to that used in the proofs of Theorem \ref{t3.1} and
\cite[Theorem 1.4]{kyz09}.

The results $\ca\dot F^s_{p,\,q}(\cx)= \dot M^s_{p,\,q}(\cx)$
and  $\ca\dot B^s_{p,\,q}(\cx)= \dot N^s_{p,\,q}(\cx)$
can be proved similarly to the proof of Theorem \ref{t3.2}.
Here we point out that the variants of Lemmas \ref{l2.1} through \ref{l2.3} still
hold in the current setting.
In fact, a variant of Lemma \ref{l2.2} is given in \cite[Lemma 4.1]{kyz09}, and
variants of Lemmas \ref{l2.3} and \ref{l2.1} can be proved by using the same ideas as
those used in the proof of \cite[Lemma 4.1]{kyz09}.
Applying these technical lemmas, via an argument as
the proofs of  (iii) and (iv) of Theorem \ref{t3.2},
we then obtain $\wz\ca\dot F^s_{p,\,q}(\cx)= \dot M^s_{p,\,q}(\cx)$
and  $\wz\ca\dot B^s_{p,\,q}(\cx)= \dot N^s_{p,\,q}(\cx)$, which completes the proof of Theorem \ref{t4.1}.
\end{proof}

\begin{rem}\label{r4.3}\rm
Assume that $\mu(\cx)=2^{-k_0}$  for some $k_0\in\zz$.
Based on Theorem \ref{t4.1},
we simply write $\ca\dot B^s_{p,\,q}(\cx)$ as $\dot B^s_{p,\,q}(\cx)$
and $\ca\dot F^s_{p,\,q}(\cx)$ as $\dot F^s_{p,\,q}(\cx)$.
This is also reasonable in the sense that $\ca\dot B^s_{p,\,p}(\cx)$
and $\ca\dot F^s_{p,\,p}(\cx)$  coincide
with $\dot \cb^s_{p}(\cx)$ when $s\in(0,\,1)$ and $p\in(1,\,\fz)$; see Proposition \ref{t4.1} below.
It is still unknown in this case if $ \dot B^s_{p,\,q}(\cx)$ and $  \dot F^s_{p,\,q}(\cx)$
can be characterized via radial Littlewood-Paley functions.
\end{rem}

Finally, we establish an equivalence between   $\dot M^s_{p,\,p}(\cx)$
and the Besov space $\dot\cb^s_p(\cx)$ considered by Bourdon and Pajot \cite{bp03} as follows.
For the characterizations of Besov and Triebel-Lizorkin spaces via differences on metric measure spaces, see
\cite{my09,gks09}.
\begin{defn}\label{d4.5}\rm
 Let $s\in(0,\,\fz)$ and $p \in[1,\,\fz)$.
Denote by $\dot \cb ^s_{p}(\cx)$ the {\it space of all $u\in L^p_\loc(\cx)$ satisfying that}
$$\|u\|_{\dot\cb ^s_{p}(\cx)}\equiv\lf(\int_\cx\int_\cx
\frac{|u(x)-u(y)|^p}{[d(x,\,y)]^{sp}V(x,\,y)}\,d\mu(y)\,d\mu(x)\r)^{1/p}<\fz.$$
\end{defn}

\begin{prop}\label{p4.1} Let $s\in(0,\,\fz)$ and $p\in[1,\,\fz)$. Then
$\dot\cb ^s_{p}(\cx)=\dot M^s_{p,\,p}(\cx)$.
\end{prop}

\begin{proof}
Let $u\in \dot\cb^s_p(\cx)$.
We need to find a fractional $s$-Haj\l asz gradient of $u$.
If $s\in(0,\,1]$, we can use the grand maximal function as in the proof of Theorem \ref{t1.2},
namely, use the equivalence $\dot M^s_{p,\,p}(\cx)=\ca\dot F^s_{p,\,p}(\cx)$ given in Theorem \ref{t4.1}.
But, for $s>1$,  we need to find another fractional $s$-Haj\l asz gradient of $u$.
Indeed, we deal with both cases in a uniform way
by taking another fractional $s$-Haj\l asz gradient.

By the reverse doubling property of the RD-space,
there exists $K_0\in\nn$ and $K_0>1$ such that for all $x\in\cx$ and  $0<r<2\diam\cx/2^{K_0}$,
$\mu(B(x,\,2^{K_0}r))\ge2\mu(B(x,\,r)).$
Notice that $u\in L^p_\loc(\cx)\subset L^1_\loc(\cx)$. Thus, for all the
Lebesgue points $x$ of $u$  and
 all $k\in\zz$ such that $2^{-k+K_0}<\diam\cx$, by the H\"older inequality, we have that
\begin{eqnarray}\label{e4.6}
&&|u(x)-u_{B(x,\,2^{-k})}|\\
&&\quad\le\sum_{j\ge k}|u_{B(x,\,2^{-j })}-u_{B(x,\,2^{-j-1})}|\nonumber \\
&&\quad\ls \sum_{j\ge k}\bint_{B(x,\,2^{-j})}|u(y)-u_{B(x,\,2^{-j+K_0+1})\setminus B(x,\,2^{-j+1})}|\,d\mu(y)\nonumber \\
&&\quad\ls \sum_{j\ge k}\bint_{B(x,\,2^{-j})}\bint_{B(x,\,2^{-j+K_0+1})\setminus B(x,\,2^{-j+1})} |u(y)-u(z)|\,d\mu(z)\,d\mu(y)\nonumber \\
&&\quad\ls \sum_{j\ge k}2^{-js}\nonumber\\
&&\quad\quad\times\lf\{\bint_{B(x,\,2^{-j})}\int_{B(x,\,2^{-j+K_0+1})\setminus B(x,\,2^{-j+1})}  \frac{|u(y)-u(z)|^p}{[d(y,\,z)]^{sp}V(y,\,z)}\,d\mu(z)\,d\mu(y)\r\}^{1/p}.\nonumber
\end{eqnarray}

If $\mu(\cx)=\fz$, for all $j\in\zz$ and $x\in\cx$, we let $$
h_j(x)\equiv \lf\{\bint_{B(x,\,2^{-j-1})}\int_{B(x,\,2^{-j+K_0+1})\setminus B(x,\,2^{-j+1})} \frac{|u(y)-u(z)|^p}{[d(y,\,z)]^{sp}V(y,\,z)}\,d\mu(z)\,d\mu(y)\r\}^{1/p},$$
and $\vec h\equiv\{h_j\}_{j\in\zz}$.
Then  $\vec h\in\wz\bd^{s,\,s,\,1}(u)$.
Let $y$ also be a Lebesgue point of $u$ with $2^{-k-1}\le d(x,\,y)<2^{-k}$.
Now 
$$|u(x)-u(y)|\le |u(x)-u_{B(x,\,2^{-k})}|+|u(y)-u_{B(x,\,2^{-k})}|.$$
Observe that by \eqref{e4.6} and an argument similar to it, we have
\begin{eqnarray*}
 |u(y)-u_{B(x,\,2^{-k })}|&&\le  |u(y)-u_{B(y,\,2^{-k+1})}|+|u_{B(y,\,2^{-k+1})}-u_{B(x,\,2^{-k })}|\\
&&\ls  \sum_{j\ge k-1}2^{-js}h_j(y) +\bint_{B(y,\,2^{-k+1})}|u(z)-u_{B(y,\,2^{-k+1})}|\,dz\\
&&\ls \sum_{j\ge k-1}2^{-js}h_j(y).
\end{eqnarray*}
So
\begin{equation}\label{e4.7}
 |u(x)-u(y)|\ls \sum_{j\ge k-1}2^{ -j s}[h_j(x)+h_j(y)].
\end{equation}
Thus, by Theorem \ref{t2.1}, we obtain that
\begin{eqnarray*}
&&\|u\|_{\dot M^s_{p,\,p}(\cx)}\\
&&\quad
 \sim\int_\cx \sum_{j\in\zz}\bint_{B(x,\,2^{-j})} \int_{B(x,\,2^{-j+K_0+1})\setminus B(x,\,2^{-j+1})} \frac{|u(y)-u(z)|^p}{[d(y,\,z)]^{sp}V(y,\,z)}\,d\mu(z) \,d\mu(y)\,d\mu(x)\\
&&\quad \ls  \sum_{j\in\zz}\int_\cx\int_{B(y,\,2^{-j+K_0+2})\setminus B(y,\,2^{-j})}
\frac{|u(y)-u(z)|^p}{[d(y,\,z)]^{sp}V(y,\,z)}\,d\mu(z) \,d\mu(y)\\
&&\quad \ls   \int_\cx\int_\cx
\frac{|u(y)-u(z)|^p}{[d(y,\,z)]^{sp}V(y,\,z)}\,d\mu(z) \,d\mu(y)\sim \|u\|_{\dot\cb^s_p(\cx)}.
\end{eqnarray*}

Now assume that $\mu(\cx)=2^{-k_0}$ for some $k_0\in\zz$.
Let $x,\,y $ be a pair of Lebesgue points of $u$ and assume that
$2^{-k-1}\le d(x,\,y)<2^{-k}$ for some $k\ge k_0$.
If $k\ge k_0+K_0+1$, then  \eqref{e4.7} still holds.
If $k_0\le k< k_0+K_0+1$, then
\begin{eqnarray}\label{e4.8}
 |u(x)-u(y)|
&& \ls |u(x)-u_{B(x,\,2^{-k_0-K_0-3})}|+|u(y)-u_{B(y,\,2^{-k_0-K_0-3})}|\\
&& \quad+|u_{B(x,\,2^{-k_0-K_0-3})}-u_{B(y,\,2^{-k_0-K_0-3})}|.\nonumber
\end{eqnarray}
Since, for all $z\in B(x,\,2^{-k_0-K_0-3})$
and $w\in B(y,\,2^{-k_0-K_0-3})$, $2^{-k}\gs d(z,\,w)\ge 2^{-k_0-K_0-3}$, we have that
\begin{eqnarray}\label{e4.9}
&&| u_{B(x,\,2^{-k_0-K_0-3})}-u_{B(y,\,2^{-k_0-K_0-3})}|\\
&&\quad\ls \bint_{B(x,\,2^{-k_0-K_0-3})}\bint_{B(y,\,2^{-k_0-K_0-3})}|u(z)-u(w)|\,d\mu(z)\,d\mu(w)\nonumber \\
&&\quad\ls 2^{-ks}\lf\{\bint_{B(x,\,2^{-k_0-K_0-3})}\int_{B(y,\,2^{-k_0-K_0-3})}
\frac{|u(z)-u(w)|^p}{[d(z,\,w)]^{sp}V(z,\,w)}\,d\mu(z)\,d\mu(w)\r\}^{1/p}\nonumber\\
&&\quad\ls 2^{-ks}[\mu(\cx)]^{1/p}\|u\|_{\dot \cb^s_p(\cx)}.\nonumber
\end{eqnarray}
If we take $h_k\equiv [\mu(\cx)]^{1/p}\|u\|_{\dot \cb^s_p(\cx)}$ for all $k_0-1\le k< k_0+K_0$
and $h_k\equiv0$ for $k<k_0-1$,
then by \eqref{e4.6}, \eqref{e4.8} and \eqref{e4.9}, we know that \eqref{e4.7} still holds and hence
   $ \vec h\equiv\{h_k\}_{k\in\zz}\in\wz\bd^{s,\,s,\,1}(u)$.
Moreover, similarly to the case $\mu(\cx)=\fz$, we have $u\in\dot M^s_{p,\,p}(\cx)$ and
$$
\|u\|_{\dot M^s_{p,\,p}(\cx)}\ls\|\vec h\|_{L^p(\cx,\,\ell^p)}\ls\|u\|_{\dot \cb^s_p(\cx)}.$$

Conversely, let $u\in \dot M^s_{p,\,p}(\cx)$. We then have that for all $x\in\cx$,
\begin{eqnarray*}
\int_\cx
\frac{|u(x)-u(y)|^p}{[d(x,\,y)]^{sp}V(x,\,y)}\,d\mu(y)
&&\ls\sum_{k=k_0}^\fz \int_{B(x,\,2^{-j})\setminus B(x,\,2^{-j-1})}
\frac{|u(x)-u(y)|^p}{[d(x,\,y)]^{sp}V(x,\,y)}\,d\mu(y)\\
&&\ls\sum_{k=k_0}^\fz \bint_{B(x,\,2^{-j}) } [g_j(x)+g_j(y)]^p\,d\mu(y),
\end{eqnarray*}
which implies that $u\in L^p_\loc(\cx)$ and
\begin{eqnarray*}
 \|u\|^p_{\dot \cb^s_p(\cx)}&&\ls\int_\cx \sum_{k=k_0}^\fz \bint_{B(x,\,2^{-j}) }
[g_j(x)+g_j(y)]^p\,d\mu(y) \,d\mu(x)\\
&&\ls\int_\cx \sum_{k=k_0}^\fz
[g_j(y) ]^p  \,d\mu(y)
\ls \|u\|^p_{\dot M ^s_{p,\,p}(\cx)}.
\end{eqnarray*}
This finishes the proof of Proposition \ref{p4.1}.
\end{proof}

\section{Quasiconformal and quasisymmetric mappings}\label{s5}

The aim of this section is to prove Theorem \ref{t1.3}, Theorem \ref{t1.4}
and their following extension; also see Corollary \ref{c5.2}.

\begin{thm}\label{t5.1}
Let $\cx$ and $\cy$ be Ahlfors $n_1$-regular and $n_2$-regular spaces with $n_1,\,n_2\in(0,\,\fz)$, respectively.
Let  $f$ be a quasisymmetric mapping from $\cx$ onto $\cy$.
For $s_i\in(0,\,n_i)$ with $i=1,\,2$, if $n_1/s_1=n_2/s_2$, then
  $f$  induces an equivalence between $\dot M^{s_1}_{n_1/s_1,\,n_1/s_1}(\cx)$
and $\dot M^{s_2}_{n_2/s_2,\,n_2/s_2}(\cy)$, and hence between $\dot \cb^{s_1}_{n_1/s_1}(\cx)$
and $\dot \cb^{s_2}_{n_2/s_2}(\cy)$.
\end{thm}

Since  the volume derivative of a quasisymmetric mapping need not satisfy the reverse H\"older inequality in this generality,
 we cannot extend  Theorem \ref{t5.1} to the full range $q\in(0,\,\fz]$.
Furthermore, we do not claim that $f$ acts as a composition operator but merely that 
every $u\in\dot\cb^{s_2}_{n_2/s_2}(\cy)$ has a representative $\wz u$ so that 
$\wz u\circ f\in\dot \cb^{s_1}_{n_1/s_1}(\cx)$ with a norm bound, and similarly for $f^{-1}$.
Indeed, $u\circ f$ need not even be measurable in this generality.

Now we begin with the proof of Theorem \ref{t1.3}.
To this end, we need the following properties of  quasiconformal mappings on $\rn$.

First recall that a homeomorphism on $\rn$
 is quasiconformal according to the metric definition  if and only if
it is quasiconformal according to the analytic definition, and if and only if it is quasisymmetric; see,
for example, \cite{hk98,k09}.
Moreover, denote by $\cb_r(\cx)$ the {\it class of functions $w$ on the metric measure space $\cx$
satisfying the reverse H\"older inequality of
order $r\in(1,\,\fz]$}: there exists a positive constant $C$ such that for all balls $B\subset\cx$,
$$\lf\{\bint_B [w(x)]^r\,d\mu(x)\r\}^{1/r}\le C\bint_Bw(x)\,d\mu(x).$$
Then, a celebrated result of Gehring \cite{g73} says that
\begin{prop}\label{p5.1}
 Let $n\ge2$  and $f:\ \rn\to \rn$ be a quasiconformal mapping. Then
there exists $r\in(1,\,\fz]$ such that $|J_f|\in\cb_r(\rn)$.
\end{prop}

For a quasiconformal mapping $f:\, \rn\to\rn$, we  set
\begin{equation}\label{e5.1}
 R_f\equiv\sup\{r\in(1,\,\fz]:\, |J_f|\in\cb_r(\rn)\}.
\end{equation}
Notice that $|J_f|\in\cb_r(\rn)$ implies that $|J_f|$ is a weight in the sense of Muckenhoupt.
Then,  we have  the following conclusions;
see, for example, \cite[Remark 6.1]{k09}.

\begin{prop}\label{p5.2}
Let $n\ge2$  and $f:\ \rn\to \rn$ be a quasiconformal mapping.

(i) For any measurable set $E\subset \rn$, $|f(E)|=\int_E|J_f(x)|\,dx$;  moreover, $|E|=0$ if and only if $|f(E)|=0$.

(ii) $f$ induces a doubling measure on $\rn$, namely,
there exists a positive constant $C$ such that  for every ball $B\subset\rn$,
 $ |f(2B )|\le C |f(B)|.$

(iii) There exist positive constants $C$ and $\az\in(0,\,1]$ such that for every ball $B\subset\rn$,
and every measurable set $E\subset B$,
$$\frac{|f(E)|}{|f(B)|}\le C\lf(\frac{|E|}{|B|}\r)^\az.$$
\end{prop}

We also need the following change of variable formula, 
which is deduced from the Lebesgue-Radon-Nikodym theorem and the absolute continuity of $f$ given in Proposition \ref{p5.2}(i).
\begin{lem}\label{l5.1}
Let  $n\ge 2$ and $f:\ \rn\to\rn$ be a quasiconformal mapping.
Then for all
nonnegative Borel measurable functions $u$ on $\rn$,
$$\int_\rn u(f(x))|J_f(x)|\,dx= \int_{\rn} u(y)\,dy.$$
\end{lem}

Let $f$ be a homeomorphism  between metric spaces $(\cx,\,d_\cx)$ and $(\cy,\,d_\cy)$.
For our convenience, in what follows, we
always write
$$L_f(x,\,r)\equiv\sup\{d_\cy(f(x),\,f(y)):\ d_\cx(x,\,y)\le r\}$$
and
$$\ell_f(x,\,r)\equiv\inf\{d_\cy(f(x),\,f(y)):\ d_\cx(x,\,y)\ge r\}$$
for all $x\in \cx$ and $r\in(0,\,\fz)$.

\begin{proof}[Proof of Theorem \ref{t1.3}]
Since $f^{-1}$ is  also a quasiconformal mapping,  it suffices to prove that
$f$ induces a bounded linear operator on $\dot M^{s}_{n/s,\,q}(\rn)$,
namely, if $u\in \dot M^{s}_{n/s,\,q}(\rn)$, then $u\circ f\in \dot M^{s}_{n/s,\,q}(\rn)$ and
 $\|u\circ f\|_{\dot M^{s}_{n/s,\,q}(\rn)}
 \ls\|u\|_{\dot M^{s}_{n/s,\,q}(\rn)}$.
To this end, let $u\in\dot M^s_{n/s,\,q}(\rn)$. Without loss of generality, we may assume that
$\|u\|_{\dot M^s_{n/s,\,q}(\rn)}=1$.
Let  $\vec g\in \bd^{s}(u)$
and $\|\vec g\|_{ L^{n/s}(\rn,\,\ell^q) }\le2$.
For our convenience, by abuse of notation,
we set $g_t\equiv g_k$ for all $t\in[2^{-k-1},\,2^{-k})$ and $k\in\zz$.
Moreover, since  either $J_f(x)>0$ for almost all $x\in\rn$ or $J_f(x)<0$ for almost all $x\in\rn$ (see, for example,
\cite[Remark 5.2]{k09}),
  without loss of generality, we may further assume that $J_f(x)>0$ for almost all $x\in\rn$.

Due to Theorem \ref{t2.1}, the task of the proof of Theorem \ref{t1.3} is reduced to finding a suitable
$\vec h\in \wz \bd^{s,\,s,\,N}(u\circ f)$  with
 $\|\vec h\|_{ L^{n/s}(\rn,\,\ell^q) }\ls1$ for some integer $N$.
To this end, we consider the following three cases: (i)   $q=n/s$,
(ii)  $q\in(n/s,\,\fz]$,
(iii) $q\in(0,\,n/s)$.
We pointed out that in Case (i), we only use the above basic properties of quasiconformal mappings in Propositions \ref{p5.1}
and \ref{p5.2} and Lemma \ref{l5.1};
 in Case (ii), we need the reverse H\"older inequality;
 while in Case (iii),  we apply Lemma \ref{l2.3} and
the reverse H\"older inequality, and establish
a subtle pointwise estimate via non-increasing rearrangement functions (see \eqref{e5.7} below).

  {\it Case (i)  $q=n/s$.}
In this case, by Proposition \ref{p5.2}(iii), there exists $K_0\in\nn$ such that for all $x\in\rn$ and $r\in(0,\,\fz)$,
$$|f(B(x,\,2^{K_0}r)\setminus B(x,\,2r))|\ge |f(B(x,\,2r))|.$$
Then for all $x\in\rn$ such that $f(x)$ is a Lebesgue point of $u$,
and for all $k\in\zz$, similarly to the proof of Lemma \ref{l2.1},
by Proposition \ref{p5.2}(ii),
we have that
\begin{eqnarray*}
 |u\circ f(x)-u_{f(B(x,\,2^{-k}))}|
&&\le \sum_{j\ge k}|u_{f(B(x,\,2^{-j-1}))}-u_{f(B(x,\,2^{-j}))}|\\
&&\ls \sum_{j\ge k}\bint_{f(B(x,\,2^{-j}))}|u(y)-u_{f(B(x,\,2^{-j+K_0})\setminus B(x,\,2^{-j+1}))}|\,dy\\
&&\ls \sum_{j\ge k}\bint_{f(B(x,\,2^{-j}))}\bint_{f(B(x,\,2^{-j+K_0})\setminus B(x,\,2^{-j+1}))} |u(y)-u(z)|\,dz\,dy,
\end{eqnarray*}
where, in the penultimate inequality, we used the fact that
$$|f(B(x,\,2^{-j-1}))|\sim|f(B(x,\,2^{-j}))|,$$
which is obtained by Proposition \ref{p5.2}(ii).

Since $f$ is a quasisymmetric mapping,  there exists $K_1\in\nn$
such that for all $y\in\rn$ and $j\in\zz$,
$$L_f(f^{-1}(y),\,2^{-j+K_0+1})\le 2^{K_1} \ell_f(f^{-1}(y),\,2^{-j})\le2^{K_1} L_f(f^{-1}(y),\,2^{-j}) .$$
For all $k\in\zz$, set  $\wz g_k\equiv\sum_{j=k}^{k+K_1} g_j$.
Then we know that $\{\wz g_k\}_{k\in\zz}\in\bd^{s,\,K_1,\,0}(u)$  and, moreover,
 $$\|\{\wz g_k\}_{k\in\zz}\|_{L^{n/s}(\rn,\,\ell^{n/s})}\ls
\|\vec g\|_{L^{n/s}(\rn,\,\ell^{n/s})}\ls1.$$
By abuse of notation, we write that $\wz g_t\equiv\wz g_k$ for every $t\in[2^{-k-1},\,2^{-k})$ and all $k\in\zz$.

For almost all $y\in f(B(x,\,2^{-j}))$ and $z\in f(B(x,\,2^{-j+K_0})\setminus B(x,\,2^{-j+1}))$,
since
$$\ell_f(f^{-1}(y),\,2^{-j})\le |y-z|\le L_f(f^{-1}(y),\,2^{-j+K_0+1}),$$
$$\ell_f(f^{-1}(z),\,2^{-j})\le|y-z|\le L_f(f^{-1}(z),\,2^{-j+K_0+1}) $$
and
$$|y-z|\le|y-f(x)|+|f(x)-z|\le 2L_f(x,\,2^{-j+K_0}),$$
 we have
\begin{eqnarray*}
 |u(y)-u(z)|&&\le |y-z|^s [g_{|y-z|}(y)+g_{|y-z|}(z)]\\
&&\ls [L_f(x,\,2^{-j+K_0})]^s[\wz g_{L_f(f^{-1}(y),\,2^{-j+K_0+1})}(y)+\wz g_{L_f(f^{-1}(z),\,2^{-j+K_0+1})}(z)],
\end{eqnarray*}
which further yields that
\begin{eqnarray*}
 &&|u\circ f(x)-u_{f(B(x,\,2^{-k}))}|\ls \sum_{j\ge k-K_0-1}[L_f(x,\,2^{-j })]^s\bint_{f(B(x,\,2^{-j }))}\wz g_{L_f(f^{-1}(y),\,2^{-j })}(y)\,dy.\end{eqnarray*}
For all $x\in\rn$ and all $j\in\zz$, set
\begin{equation}\label{e5.2}
h_j(x)\equiv 2^{js}[L_f(x,\,2^{-j })]^s\bint_{f(B(x,\,2^{-j }))}\wz g_{L_f(f^{-1}(y),\,2^{-j })}(y)\,dy.
\end{equation}
Then $$ |u\circ f(x)-u_{f(B(x,\,2^{-k}))}|\ls \sum_{j\ge k-K_0-1} 2^{-js}h_j(x).$$
Moreover, $\vec h$ is a constant multiple of an element of $\wz \bd^{s,\,s,\,K_0+2}(u\circ f)$.
In fact, for every pair of Lebesgue points $x,\,y\in\rn$ with $|x-y|\in[2^{-k-1},\,2^{-k})$,
we have
\begin{equation*}
 |u\circ f(x)-u\circ f(y)|
\le |u\circ f(x)-u_{f(B(x,\,2^{-k }))}|+ |u\circ f(y)-u_{f(B(x,\,2^{-k }))}|.
\end{equation*}
By Proposition \ref{p5.2}(ii) and an argument similar to the above, we also have
\begin{eqnarray*}
&&|u\circ f(y)-u_{f(B(x,\,2^{-k+1}))}|\\
&&\quad\ls  |u\circ f(y)-u_{f(B(y,\,2^{-k+1}))}|
+|u_{f(B(y,\,2^{-k+1 }))}-u_{f(B(x,\,2^{-k}))}|\\
&&\quad\ls  \sum_{j\ge k-K_0-2} 2^{-js}h_j(y) + \bint_{f(B(y,\,2^{-k+1}))}|u(z)-u_{f(B(y,\,2^{-k+1 }))}|\,dz\\
&&\quad\ls  \sum_{j\ge k-K_0-2} 2^{-js}h_j(y),
\end{eqnarray*}
and hence
\begin{equation}\label{e5.3}
 |u\circ f(x)-u\circ f(y)|\ls \sum_{j\ge k-K_0-2}2^{-js}[h_j(x)+h_j(y)],
\end{equation}
which  implies that $\vec h$ is a constant
multiple of an element of $\wz \bd^{s,\,s,\,K_0+2}(u\circ f)$.

Now we estimate $\|\vec h\|_{L^{n/s}(\rn,\,\ell^{n/s})}$.
In fact, from (ii) and (iii) of Proposition \ref{p5.2} and the fact that $f$ is quasisymmetric, it follows that
for all $x\in\rn$ and $j\in\zz$,
\begin{equation}\label{e5.4}
[L_f(x,\,2^{-j})]^n\sim |f(B(x,\,2^{-j}))|,
\end{equation}
which together with the H\"older inequality implies that
$$  h_j(x)\ls 2^{js}\lf\{\int_{f(B(x,\,2^{-j }))}[\wz g_{L_f(f^{-1}(y),\,2^{-j })}(y)]^{n/s}\,dy\r\}^{s/n}.
$$ Noticing that $y\in f(B(x,\,2^{-j}))$ implies that $x\in B(f^{-1}(y),\,2^{-j})$,
 by Proposition \ref{p5.2}(i),  we  have
\begin{eqnarray*}
 \|\vec h\|^{n/s}_{L^{n/s}(\rn,\,\ell^{n/s}) }
&& \ls \sum_{j\in\zz}2^{jn} \int_\rn  \int_{f(B(x,\,2^{-j }))}[\wz g_{L_f(f^{-1}(y),\,2^{-j })}(y)]^{n/s}\,dy\,dx\\
&& \ls \sum_{j\in\zz}2^{jn}\int_\rn[\wz g_{L_f(f^{-1}(y),\,2^{-j })}(y)]^{n/s}\lf\{\int_{B(f^{-1}(y),\,2^{-j })}\,dx\r\}\,dy\\
&& \ls \int_\rn\sum_{j\in\zz} [\wz g_{L_f(f^{-1}(y),\,2^{-j})}(y)]^{n/s} \,dy\\
&& \ls \int_\rn\sum_{k\in\zz}\lf(\sharp\lf\{j\in\zz:\,  L_f(f^{-1}(y),\,2^{-j })\in[2^{-k-1},\,2^{-k})\r\}\r)
 [\wz g_k(y)]^{n/s} \,dy,
\end{eqnarray*}
where $\sharp E $ denotes the cardinality of a set $E\subset \zz$.
Moreover, observe that  for all $k\in\zz$ and $y\in\rn$, we have
\begin{equation}\label{e5.5}
 \sharp\lf\{j\in\zz:\  L_f(f^{-1}(y),\,2^{-j})\in[2^{-k-1},\,2^{-k}) \r\}\ls 1.
\end{equation}
Indeed, if $i,\,j\in\zz$ with $i>j$,
$L_f(f^{-1}(y),\,2^{-i}),\, L_f(f^{-1}(y),\,2^{-j})\in[2^{-k-1},\,2^{-k})$,
then by \eqref{e5.4} and Proposition \ref{p5.2}(iii),
$$\frac12\le \frac {L_f(f^{-1}(y),\,2^{-i}) }
{L_f(f^{-1}(y),\,2^{-j}) }\ls \frac{|f(B(f^{-1}(y),\,2^{-i} ))|^{1/n}}{|f(B(f^{-1}(y),\,2^{-j} ))|^{1/n}}
\ls 2^{(j-i)\az/n},$$
which implies that  $i-j\le N$ for some constant N independent of $i$, $j$ and $y$,
and hence \eqref{e5.5} follows.
Then by \eqref{e5.5}, we further obtain
\begin{eqnarray*}
\|\vec h\|^{n/s}_{L^{n/s}(\rn,\,\ell^{n/s})}
 \ls \sum_{k\in\zz}\int_\rn [\wz g_k(y)]^{n/s} \,dy\ls \|\vec g\|^{n/s}_{L^{n/s}(\rn,\,\ell^{n/s})}\ls1,
\end{eqnarray*}
which implies that $\|u\circ f\|_{\dot M^s_{n/s,\,n/s}(\rn)}\ls1$.
That is, Theorem \ref{t1.3} is true for the space $\dot M^s_{n/s,\,n/s}(\rn)$.

{\it  Case (ii)  $q\in(n/s,\,\fz]$.} In this case, we still take $\vec h\equiv\{h_j\}_{j\in\zz}$
as a variant of the fractional $s$-Haj\l asz gradient of $u\circ f$,
where $h_j$ is given in \eqref{e5.2}.
Then we will control $h_j$
by a suitable maximal function   via an application of the reverse H\"older inequality satisfied by $J_f$.
In fact, by Lemma \ref{l5.1}, \eqref{e5.4} and  the H\"older inequality, we have
\begin{eqnarray*}
h_j(x)&&= \lf[\frac{|f(B(x,\,2^{-j })|}{|B(x,\,2^{-j })|} \r]^{-1+s/n}
\bint_{B(x,\,2^{-j })}\wz g_{L_f(z,\,2^{-j })}(f(z))J_f(z)\,dz \\
&&\ls\lf[\frac{|f(B(x,\,2^{-j })|}{|B(x,\,2^{-j })|} \r]^{-1+s/n}
\lf\{\bint_{B(x,\,2^{-j})} \lf[\wz g_{L_f(z,\,2^{-j })}(f(z))\r]^p[J_f(z)]^{ps/n}\,dz \r\}^{1/p}\\
&&\quad \times\lf\{\bint_{B(x,\,2^{-j})}[J_f(z)]^{p(n-s)/[n(p-1)]}\,dz \r\}^{(p-1)/p},
\end{eqnarray*}
where we take $p\in(1,\,n/s)$ to be close to $n/s$ so that $p(n-s)/n(p-1)< R_f$.
Therefore, by  the
reverse H\"older inequality given in Proposition \ref{p5.1}, and Proposition \ref{p5.2}(i), we obtain
 $$
h_j(x) \ls \lf\{\cm\lf( \lf[\wz g_{L_f(\cdot,\,2^{-j })}\circ f  \r]^p[J_f ]^{ps/n}\r)(x)\r\}^{1/p},$$
where we recall that $\cm$ denotes the Hardy-Littlewood maximal function.
Therefore,  the Fefferman-Stein vector-valued maximal inequality on $\cm$, $p<n/s<q$, \eqref{e5.5} and Lemma \ref{l5.1} yield that
\begin{eqnarray*}
\|\vec h\|^{n/s}_{L^{n/s}(\rn,\,\ell^q) }
&&\ls \int_\rn \lf(\sum_{j\in\zz}
\lf\{\cm\lf(\lf[\wz g_{L_f(\cdot,\,2^{-j })}\circ f  \r]^p[J_f ]^{ps/n}\r)(x)\r\}^{q/p}\r)^{n/(sq)}\,dx\\
&&\ls
\int_\rn \lf(\sum_{j\in\zz} \lf[\wz g_{L_f(x,\,2^{-j })}\circ f(x)\r]^q \r)^{n/(sq)}J_f (x)\,dx\\
&&\ls
\int_\rn \lf(\sum_{j\in\zz} \lf[\wz g_{j}\circ f(x)\r]^q \r)^{n/(sq)}J_f (x)\,dx\\
&&\ls
\int_\rn \lf(\sum_{j\in\zz} \lf[\wz g_j (x)\r]^q \r)^{n/(sq)} \,dx\ls1.
\end{eqnarray*}
Thus, Theorem \ref{t1.3} is true for the space $\dot M^s_{n/s,\,q}(\rn)$ with $q\in(n/s,\,\fz]$.

{\it Case (iii) $q\in(0,\,n/s)$.} In this case,  
for given $q\in(n/(n+s), n/s)$, we choose $\dz \in(0,\,1]$ such that
$$
0<\dz < \frac{nq(R_f-1)}{nR_f-sq},$$
where $R_f$ is as in \eqref{e5.1} on $\rn$.
It is easy to check that
$$\frac{n-s\dz }n\frac{q}{q-\dz }<R_f.$$
Observe that we can take $p\in(1,\,q/\dz)$ and close to $q/\dz$ such that
\begin{equation}\label{e5.6}
 \frac{n-s\dz }n\frac{p}{p-1}<R_f.
\end{equation}
We also let $\ez,\,\ez'\in(\max\{n(q-1)/q,\,0\},\,s)$ such that $\ez<\ez'$.

We now claim that there exists a measurable set $E\subset\rn$ with $|E|=0$
such that for all $x,\,y\in \rn\setminus E$ with $|x-y|\in[2^{-k-1}, 2^{-k})$,
\begin{eqnarray}\label{e5.7}
|u\circ f(x)-u\circ f(y)|
&&\ls \sum_{j\ge k} \inf_{c\in\rr}\lf(\bint_{B(f(x),\, 2L_f(x,\,2^{-j}))}|u(y)-c|^\dz \,dy\r)^{1/\dz }\\
&&\quad\quad+\sum_{j\ge k} \inf_{c\in\rr}\lf(\bint_{B(f(y),\, 2L_f(y,\,2^{-j}))}|u(y)-c|^\dz \,dy\r)^{1/\dz },\nonumber
\end{eqnarray}
where the implicit constant is independent of $x,\,y,\,k$ and $u$, but may depend on $\dz$.

Assume this claim holds for the moment.
Observe that by Proposition \ref{p5.2}(iii) and \eqref{e5.4}, there exists  $K_2\in\nn$ such that  for all $x\in\rn$ and $j\in\zz$,
$4L_f(x,\,2^{-j})\le \ell_f(x,\,2^{-j+K_2})$, and hence
$$B(f(x),\, 4L_f(x,\,2^{-j}))\subset B(f(x),\, \ell_f(x,\,2^{-j+K_2}))\subset f(B(x,\,2^{-j+K_2})).$$
Then by Lemma \ref{l2.3}, we have that
\begin{eqnarray*}
&&\sum_{j\ge k} \inf_{c\in\rr}\lf(\bint_{B(f(x),\, 2L_f(x,\,2^{-j}))}|u(y)-c|^\dz \,dy\r)^{1/\dz }\\
&&\quad\ls \sum_{j\ge k} [L_f(x,\,2^{-j})]^{\ez'}
\sum_{i\ge -3-\log L_f(x,\,2^{-j})} 2^{-i(s-\ez')} \lf\{\bint_{B(f(x),\,  L_f(x,\,2^{-j+K_2}))}[g_i(z)]^\dz \,dz\r\}^{1/\dz }.
\end{eqnarray*}
Notice that by \eqref{e5.6}, the reverse H\"older inequality given in Proposition \ref{p5.1},
the H\"older inequality,  Proposition \ref{p5.2}(i) 
and   Lemma \ref{l5.1},
we obtain that
\begin{eqnarray*}
&& \bint_{f(B(x,\,2^{-j+K_2}))}[g_i(z)]^\dz \,dz\\
&&\quad\ls \frac{|B(x,\,2^{-j+K_2}))|}{|f(B(x,\,2^{-j+K_2}))|}\bint_{ B(x,\,2^{-j+K_2}) }[g_i\circ f(y)]^\dz  J_f(y)\,dy\\
&&\quad\ls \frac{|B(x,\,2^{-j+K_2}))|}{|f(B(x,\,2^{-j+K_2}))|}
\lf\{\bint_{ B(x,\,2^{-j+K_2}) }[g_i\circ f(y)]^{p\dz} [J_f(y)]^{p\dz s/n }\,dy\r\}^{1/p}\\
&&\quad\quad\times \lf\{\bint_{ B(x,\,2^{-j+K_2}) } [J_f(y)]^{p(n -s\dz)/[n(p-1)]}\,dy\r\}^{(p-1)/p}\\
&&\quad\ls \lf[\frac{|B(x,\,2^{-j+K_2})|}{|f(B(x,\,2^{-j+K_2}))|}\r]^{\dz s/n}
\lf[\cm([g_i\circ f]^{p\dz} [J_f]^{p\dz s/n})(x)\r]^{1/p}.
\end{eqnarray*}
Thus, by \eqref{e5.4},
\begin{eqnarray*}
&&\sum_{j\ge k} \inf_{c\in\rr}\lf(\bint_{B(f(x),\, 2L_f(x,\,2^{-j}))}|u(y)-c|^\dz \,dy\r)^{1/\dz }\\
&&\quad\ls \sum_{j\ge k} [L_f(x,\,2^{-j})]^{\ez'}
\sum_{i\ge -3-\log L_f(x,\,2^{-j})} 2^{-i(s-\ez')}\lf[\frac{|B(x,\,2^{-j+K_2})|}{|f(B(x,\,2^{-j+K_2}))|}\r]^{s/n}\\
&&\quad\quad\times
\lf[\cm([g_i\circ f]^{p\dz} [J_f]^{p\dz s/n})(x)\r]^{1/(p\dz)}\\
&&\quad\ls \sum_{j\ge k} 2^{-js}
\sum_{i\ge -3-\log L_f(x,\,2^{-j})} 2^{-i(s-\ez')} [L_f(x,\,2^{-j})]^{\ez'-s}\\
&&\quad\quad\times
\lf[\cm([g_i\circ f]^{p\dz} [J_f]^{p\dz s/n})(x)\r]^{1/(p\dz)}.
\end{eqnarray*}
For all $j\in\zz$, set
\begin{eqnarray*}
h_j&&\equiv \sum_{i\ge -3-\log L_f(\cdot,\,2^{-j})} [L_f(\cdot,\,2^{-j})]^{\ez'-s}2^{-i(s-\ez')}
\lf[\cm([g_i\circ f]^{p\dz} [J_f]^{p\dz s/n})\r]^{1/(p\dz)}.
\end{eqnarray*}
By \eqref{e5.7},  we know that
$\vec h\equiv\{h_j\}_{j\in\zz}$ is a constant multiple of an element of $\wz\bd^{s,\,s,\,0}(u\circ f)$.
Moreover, by \eqref{e5.5},
we have
\begin{eqnarray*}
&&\|\vec h\|^{n/s}_{L^{n/s}(\rn,\,\ell^q) }\\
&&\quad\ls \int_\rn \Bigg(\sum_{j\in\zz}
\Bigg\{\sum_{i\ge -3-\log L_f(x,\,2^{-j})} [L_f(x,\,2^{-j})]^{\ez'-s}2^{-i(s-\ez')} \\
&&\quad\quad\times
\lf[\cm([g_i\circ f]^{p\dz} [J_f]^{p\dz s/n})(x)\r]^{1/(p\dz)}\Bigg\}^q\Bigg)^{n/(sq)}\,dx\\
&&\quad\ls \int_\rn \Bigg(\sum_{k\in\zz}
\Bigg\{\sum_{i\ge k-3}  2^{(k-i)(s-\ez')}
\lf[\cm([g_i\circ f]^{p\dz} [J_f]^{p\dz s/n})(x)\r]^{1/(p\dz)}\Bigg\}^q\Bigg)^{n/(sq)}\,dx.
\end{eqnarray*}
When $q\in(1,\,n/s)$, applying the H\"older inequality, 
we   have
\begin{eqnarray*}
\|\vec h\|^{n/s}_{L^{n/s}(\rn,\,\ell^q) } \ls \int_\rn \Bigg(\sum_{k\in\zz}
\sum_{i\ge k-3}  2^{(k-i)(s-\ez')}
\lf[\cm([g_i\circ f]^{p\dz} [J_f]^{p\dz s/n})(x) \r]^{q /(p\dz)}\Bigg)^{n/(sq)}\,dx,
\end{eqnarray*}
which, when $q\in(n/(n+s),\,1]$, still holds with $2^{(k-i)(s-\ez')}$ replaced by $2^{(k-i)(s-\ez')q}$
due to \eqref{e2.2}.
Then, by $p<q/\dz <n/(sr)$, the Fefferman-Stein vector-valued maximal inequality on $\cm$ and Lemma \ref{l5.1},
 we obtain
\begin{eqnarray*}
\|\vec h\|^{n/s}_{L^{n/s}(\rn,\,\ell^q) }&& \ls \int_\rn \lf(\sum_{i\in\zz}
\lf[\cm([g_i\circ f]^{p\dz} [J_f]^{p\dz s/n})(x) \r]^{q/(p\dz)}\r)^{n/(sq)}\,dx\\
&& \ls
\int_\rn \lf(\sum_{i\in\zz} \lf[g_{i}\circ f(x)\r]^q \r)^{n/(sq)}J_f (x)\,dx\\
&& \ls
\int_\rn \lf(\sum_{i\in\zz} \lf[ g_i (x)\r]^q \r)^{n/(sq)} \,dx\ls1,
\end{eqnarray*}
 which is  as desired.

Finally, we prove the above  claim \eqref{e5.7}.
For each ball $B$, let $m_u(B)$ be a median of $u$ on $B$, namely,
a real number such that
$$\max\lf\{|\{x\in B:\ f(x)>m_u(B)\}|,\,|\{x\in B:\ f(x)<m_u(B)\}|\r\} \le \frac{|B|}2.$$
Then, as proved by Fujii \cite[Lemma 2.2]{f90}, 
there exists a measurable set $E\subset \rn$ with $|E|=0$ such that 
$$u(z)=\lim_{|B|\to0,\,B\ni z}m_u(B)$$ for  all $z\in\rn\setminus E$.
Thus, for all $z\in\rn\setminus E$,   and every sequence $\{r_j\}_{j\ge0}$ with $r_j\to0$ as $j\to\fz$
and  $0<r_{j+1}\le r_j< N  r_{j+1}$ for some fixed constant $N$,
we have
\begin{eqnarray*}
|u(z)-m_u(B(z,\,r_0))|
&&\le \sum_{j\ge 0}|m_u(B(z,\,r_{j+1}))-m_u(B(z,\,r_j))|\\
 &&\le \sum_{j\ge 0}(|m_u(B(z,\,r_{j+1}))-c_{B(z,\,r_j)}|
+|m_u(B(z,\,r_{j}))-c_{B(z,\,r_j )}|),
\end{eqnarray*}
where $c_{B(z,\,r_j)}$ is a real number such that
\begin{equation}\label{e5.8}
 \bint_{B(z,\,   r_j )}|u(w)-c_{B(z,\,r_j)}|^\dz \,dw\le
2\inf _{c\in\rr} \bint_{B(z,\,  r_j)}|u(w)-c|^\dz \,dw.
\end{equation}

To estimate $|m_u(B(z,\,r_{j+1}))-c_{B(z,\,r_j)}|$ and $|m_u(B(z,\,r_{j}))-c_{B(z,\,r_j )}|$,
 recall that  the non-increasing rearrangement of a measurable function $v$ is defined by
$$v^\ast(t)\equiv \inf\{\az>0:\ |\{w\in\rn:\ |v(w)|>\az\}|\le t\}.$$
Then,  for every ball $B$ and number $c\in\rr$, obviously,
we can take $m_{u-c}(B)=m_u(B)-c$ as a median of $u-c$ on $B$. Then,
by \cite[Lemma 2.1]{f90},
$$  |m_u(B)-c|=|m_{u-c}(B)|\le m_{|u-c|}(B) ,$$
which further implies that
\begin{equation}\label{e5.9}
  |m_u(B)-c|
 \le (|u-c|\chi_B)^\ast(|B|/2)\le \lf\{2\bint_B|u(w)-c|^\dz \,dw\r\}^{1/\dz }.
 \end{equation}
Indeed, letting $\sz\equiv \bint_B|u(w)-c|^\dz\,dw$,
by Chebyshev's inequality, we have
\begin{eqnarray*}
\lf|\lf\{w\in B:\ |u(w)-c|>(2\sz)^{1/\dz }\r\}\r|
&&=\lf|\lf\{w\in B:\ |u(w)-c|^\dz> 2\sz \r\}\r|\\
&&\le ( 2\sz)^{-1 }\int_B|u(w)-c|^\dz\,dw=\frac{|B|}2,
\end{eqnarray*}
which implies the second inequality of \eqref{e5.9}.
For the first inequality, since
$$|\{w\in B:\ |u(w)-c|\ge m_{|u-c|}(B)\}|=|B|-|\{w\in B:\ |u(w)-c|< m_{|u-c|}(B)\}|\ge \frac{|B|}2,$$
for all $\az<m_{|u-c|}(B)$, we have $|\{w\in B:\ |u(w)-c|> \az\}|\ge |B|/2$, which implies that
 $\az< (|u-c|\chi_B)^\ast(|B|/2)$
and hence  $m_{|u-c|}(B)\le(|u-c|\chi_B)^\ast(|B|/2)$. This gives the first inequality of \eqref{e5.9}.

Combining \eqref{e5.9}, \eqref{e5.8} and $r_{j+1}\le r_j\le Nr_j$ yields that
\begin{eqnarray*}
 |m_u(B(z,\,r_j))-c_{B(z,\,r_j)}|
&&\ls\inf_{c\in\rr}\lf\{ \bint_{B(z,\,r_j)}|u(w)-c|^\dz \,dw\r\}^{1/\dz }
\end{eqnarray*}
and
\begin{eqnarray}\label{e5.10}
 |m_u(B(z,\,r_{j+1}))-c_{B(z,\, r_j)}|&&\le
\lf\{2\bint_{B(z,\, r_{j+1})}|u(w)-c_{B(z,\,r_j)}|^\dz \,dw\r\}^{1/\dz }\\
&&\ls\inf_{c\in\rr}\lf\{ \bint_{B(z,\, r_j)}|u(w)-c|^\dz \,dw\r\}^{1/\dz }.\nonumber
\end{eqnarray}
Therefore,
\begin{equation}\label{e5.11}
 |u(z)-m_u(B(z,\,r_0))|
 \ls \sum_{j\ge 0}\inf_{c\in\rr}\lf\{ \bint_{B(z,\,r_j)}|u(w)-c|^\dz \,dw\r\}^{1/\dz }.
\end{equation}

For  all $x,\,y\in\rn\setminus E$ with $2^{-k-1}\le |x-y|<2^{-k}$, write
\begin{eqnarray*}
 |u\circ f(x)-u\circ f(y)|
&& \le
|u\circ f(x)-m_u(B(f(x),\,2L_f(x,\,2^{-k}))|\\
&& \quad +
|m_u(B(f(y),\,L_f(y,\,2^{-k-1}))-m_u(B(f(x),\,2L_f(x,\,2^{-k}))|\\
&& \quad  +
|u\circ f(y)-m_u(B(f(y),\,L_f(x,\,2^{-k-1}))|\\
&&\equiv I_1+I_2+I_3.
\end{eqnarray*}
Since $$L_f(z,\,2^{-{j+1}})\le L_f(z,\,2^{-{j }})\le 2^{K_1}L_f(z,\,2^{-{j+1}}),$$
by \eqref{e5.11}, we have
$$ I_1
 \ls \sum_{j\ge 0}\inf_{c\in\rr}\lf\{ \bint_{B(f(x),\,2L_f(x,\,2^{-j}))}|u(w)-c|^\dz \,dw\r\}^{1/\dz }$$
and
 $$I_3 \ls
  \sum_{j\ge 1}\inf_{c\in\rr}\lf\{ \bint_{B(f(y),\,2L_f(x,\,2^{-j}))}|u(w)-c|^\dz \,dw\r\}^{1/\dz }.$$
Moreover, by an argument similar to \eqref{e5.10}, we have
\begin{eqnarray*}
I_2
&&\ls
|m_u(B(f(y),\,L_f(y,\,2^{-k-1}))-c_{B(f(x),\, 2L_f(x,\,2^{-k}))}|\\
&&\quad\quad+|m_u(B(f(x),\,2L_f(x,\,2^{-k}))-c_{B(f(x),\,2L_f(x,\,2^{-k}))}|\\
&&\ls \inf_{c\in\rr}\lf\{ \bint_{B(f(x),\,2L_f(x,\,2^{-k}))}|u(w)-c|^\dz \,dw\r\}^{1/\dz }.
\end{eqnarray*}
Combining the estimates of $I_1$, $I_2$ and $I_3$ gives the above   claim \eqref{e5.7}.
This finishes  the proof of Case (iii)
and hence  Theorem \ref{t1.3}.
\end{proof}

To prove Theorem \ref{t1.4}, we need the following result which is deduced from Theorem 7.11 and Corollary 7.13 of \cite{hk98},
\cite{kz08} and the Lebesgue-Radon-Nikodym theorem.

\begin{prop}\label{p5.3}
Let $\cx$ and $\cy$ be locally compact Ahlfors $n$-regular spaces for some $n>1$ and 
assume that 
$\cx$ admits a weak $(1,\,n)$-Poincar\'e  inequality. 
Let $f:\ \cx\to\cy$ be a quasisymmetric mapping. 
Then

(i) $f$ is absolutely continuous,
$J_f\in\cb_r(\cx)$ for some $r\in(1,\,\fz]$, $J_fd\mu_\cx$ is a doubling measure and
\begin{equation*} \int_EJ_f(x)\,d\mu_\cx(x)= |f(E)|\end{equation*}
for every measurable set $E\subset\cx$.
Here $J_f$ denotes the volume (Radon-Nikodym) derivative  of $f$, namely, 
\begin{equation*}
 J_f(x)\equiv\lim_{r\to0}\frac{\mu_\cy(f(B(x,\,r)))}{\mu_\cx(B(x,\,r))},
\end{equation*}
which exists and is finite for almost all $x\in\cx$.

(ii) $f^{-1}$ is also a quasisymmetric mapping, absolutely continuous and
for almost all $x\in\cx$, $J_{f^{-1}}(f(x))=[J_f(x)]^{-1}$ .
\end{prop}

\begin{proof}[Proof of Theorem \ref{t1.4}.]
With the assumptions of Theorem \ref{t1.4}, by Corollary 4.8 and Theorem 5.7 of \cite {hk98},
we know that $f$ is actually an $\eta$-quasisymmetric mapping
for some homeomorphism $\eta:\ [0,\,\fz)\to[0,\,\fz)$. Moreover, since the assumptions of Theorem \ref{t1.4}
also imply those of Proposition \ref{p5.3}, we have that $J_f\in\cb_r(\cx)$ and 
  $J_fd\mu_\cx$ is also a doubling measure, which together with  \cite{st89} imply that 
$J_f$ is a weight in the sense of Muckenhoupt.
Therefore, a variant of Proposition \ref{p5.2}(iii) still holds in this setting. 
Moreover, recall that by \cite[Proposition 10.8]{hei},
for every pair of sets $A$ and $B$ satisfying $A\subset B\subset\cx$ and $0<\diam A\le \diam B<\fz$,
\begin{equation}\label{e5.12}
\frac{1}{2\eta\lf(\frac{\diam B}{\diam A}\r)}\le \frac{\diam f(A)}{\diam f(B)}\le \eta\lf(\frac{2\diam A}{\diam B}\r).
\end{equation}
Then, with the aid of these facts, Proposition \ref{p5.3} and Proposition \ref{p4.1},
the proof of Theorem \ref{t1.4} is essentially the same as that of Theorem \ref{t1.3}.
We omit the details.
\end{proof}

As the above proof shows,  with the assumptions of Proposition \ref{p5.3},
a similar conclusion  of Theorem \ref{t1.4} still holds. 

\begin{cor}\label{r5.1}
Let the assumptions be as in Proposition \ref{p5.3}. 
Then for all $s\in(0,\,1]$ and $q\in(0,\,\fz]$, 
a quasisymmetric mapping $f:\ \cx\to\cy$ induces an equivalence between  $\dot M^s_{n/s,\,q}(\cx)$
and $\dot M^s_{n/s,\,q}(\cy)$.  
\end{cor}

Finally, we turn to the proof of Theorem \ref{t5.1}.
To this end, we need the
following Lebesgue theorem for Haj\l asz-Sobolev functions,
which is proved by modifying the proof of
\cite[Theorem 4.4]{hak98} slightly (see also \cite[Theorem 4.4]{kl02}).

\begin{lem}\label{l5.2}
Let $\cx$ be an Ahlfors $n$-regular space with $n>0$ and $s\in(0,\,n)$.
Then for every $u\in \dot \cb^{s}_{n/s}(\cx)$, 
there exists a set $F$ such that $\cx\setminus F$ has   Hausdorff dimension zero
and $\wz u(x)\equiv \lim_{r\to0}\bint_{B(x,\,r)} u(z)\,d\mu(z)$
exists for all $x\in F$.
\end{lem}

\begin{proof}
Let $u\in \cb^{s}_{n/s}(\cx)\subset\dot M^{s,\,n/s}(\cx)$. For all $x\in\cx$,
define
$$\wz u(x)\equiv \limsup_{r\to0}\bint_{B(x,\,r)} u(z)\,d\mu(z).$$
By Lemma \ref{l2.2}, $u$ is locally integrable and hence, $\wz u(x)=u(x)$ for almost all $x\in\cx$.
Denote by $F$ the set of all $x\in\cx$ such that  $\wz u(x)= \lim_{r\to0}\bint_{B(x,\,r)} u(z)\,d\mu(z).$
Then, to show Lemma \ref{l5.2},
it suffices to prove that $\cx\setminus F$ has  Hausdorff dimension zero.
Let $g\in\cd^s(\cx)\cap L^{n/s}(\cx)$
 and $\ez\in(0,\,s)$.
Notice that for $x\in\cx$ and $j,\,k\in\zz$ with $k\ge j+1$,
\begin{eqnarray*}
|u_{B(x,\,2^{-k})}-u_{B(x,\,2^{-j})}|&&\le \sum_{i= j}^{k-1} |u_{B(x,\,2^{-i-1})}-u_{B(x,\,2^{-i})}|\\
&&\le \sum_{i= j}^{k-1}  \bint_{B(x,\,2^{-i})}|u(z)-u_{B(x,\,2^{-i})}|\,d\mu(z)\\
&&\le \sum_{i= j}^{k-1} 2^{-is} \bint_{B(x,\,2^{-i})}g(z)\,d\mu(z)\\
&&\ls 2^{-j(s-\ez)}\sup_{i\ge j} 2^{-i\ez} \bint_{B(x,\,2^{-i})}g(z)\,d\mu(z).
\end{eqnarray*}
If $\sup_{i\ge 0} 2^{-i\ez} \bint_{B(x,\,2^{-i})}g(z)\,dz<\fz$, then  $x\in F$.
Thus,
$$\cx\setminus F\subset G \equiv \lf\{x\in\cx:\ \sup_{i\ge 0} 2^{ -i \ez} \bint_{B(x,\,2^{-i})}g(z)\,d\mu(z)=\fz\r\}.$$
Since for all $i\in\nn$ and $x\in\cx$, by the H\"older inequality,
$$2^{ -i \ez} \bint_{B(x,\,2^{-i})}g(z)\,d\mu(z)\le \lf\{2^{ -i n\ez/s }
\bint_{B(x,\,2^{-i})}[g(z)]^{n/s}\,d\mu(z)\r\}^{s/n},$$
then by \cite[Lemma 2.6]{hak98},  for all $N\in\nn$, we further have
\begin{eqnarray*}
&&\ch_\fz^{n(1-\ez/s)}(G\cap B(y,\,1))\\
&&\quad\le \ch_\fz^{n(1-\ez/s)}\lf(\lf\{x\in B(y,\,1):\ \sup_{i\ge 0} 2^{ -in \ez/s}  \bint_{B(x,\,2^{-i})}[g(z)]^{n/s}\,d\mu(z)>N^{n/s}\r\}\r)\\
&&\quad\ls N^{-n/s} \int_\cx[g(z)]^{n/s}\,d\mu(z),
\end{eqnarray*}
which implies that $\ch_\fz^{n(1-\ez/s)}(G\cap B(y,\,1))=0$ and hence $\ch^{n(1-\ez/s)}(G\cap B(y,\,1))=0$.
Here $\ch_\fz^{n(1-\ez/s)}$ and $\ch^{n(1-\ez/s)}$ denote the Hausdorff content and the Hausdorff measure,
respectively.
Because we are free to choose $\ez\in(0,\,s)$, we conclude that $G$ and
hence $\cx\setminus F$ have  Hausdorff dimension zero.
This finishes the proof of Lemma \ref{l5.2}.
\end{proof}

\begin{proof}[Proof of Theorem \ref{t5.1}]
Let $f:\ \cx\to\cy$ be an $\eta$-quasisymmetric mapping.
By \cite[Proposition 10.6]{hei},
  $f^{-1}$ is an $\wz\eta$-quasisymmetric with $\wz\eta(t)=1/\eta^{-1}(t^{-1})$ for all $t\in(0,\,\fz)$.
Moreover, since $\cx$ and $\cy$ are Ahlfors regular spaces and hence uniformly perfect, by \cite[Corollary 11.5]{hei},
$f$ and $f^{-1}$ are H\"older continuous on bounded sets of $\cx$ and $\cy$, respectively.

Let $u\in\dot M^{s_2}_{n_2/s_2,\,n_2/s_2}(\cy)$ and $\wz u$ be as in Lemma \ref{l5.2}.
Then Lemma \ref{l5.2} says that $\wz u(x)=u(x)$ for almost all $x\in\cy$,
and the complement of the set of all the Lebesgue points of $\wz u$ is contained in $\cy\setminus F$
and has Hausdorff dimension zero.
By abuse of notation, we still denote $\wz u$ by $u$. Since
$f^{-1}$ is H\"older continuous on every bounded set of $\cy$,
it is easy to check that $f^{-1}(\cy\setminus F)$ has Hausdorff dimension zero and
hence $\mu_\cx(f^{-1}(\cy\setminus F))=0$.
Let $\vec g\in\bd^{s_2}(u)$ such that
$\|\vec g\|_{L^{n_2/s_2}(\cy,\,\ell^{n_2/s_2})}\ls\|u\|_{\dot \cb^{s_2}_{n_2/s_2}(\cy)}$
and \eqref{e1.1} holds with some set $E$ having measure zero.
For all $x\in f^{-1}(F)\subset\cx$  and for all $j\in\zz$, if $2^{-j}<2\diam\cx$,   
set$$h_j(x)\equiv 2^{js_1} [L_f(x,\,2^{-j})]^{s_2} \bint_{f(B(x,\,2^{-j}))}
\wz g_{L_f(f^{-1}(y),\,2^{-j})}(y)\,d\mu_\cy(y),$$
and if  $2^{-j}\ge 2\diam\cx$,  set $h_j\equiv 0$.
Since  $\mu_\cx(\cx\setminus f^{-1}(F))=\mu_\cx( f^{-1}(\cy\setminus F))=0$, 
$\vec h$ is well-defined.
 Moreover, for each $x\in f^{-1}(F)$, since $f(x)$ is a Lebesgue point of $u$,
it follows that $u_{f(B(x,\,2^{-j}))}\to u\circ f(x)$ as $j\to\fz$.
Observing that
$\mu_\cx(\cx\setminus f^{-1}(F))=0$, by an argument as in the proof of Theorem \ref{t1.3}, we can prove that
$\vec h\equiv\{h_j\}_{j\in\zz}$ is a constant multiple of an element of $\bd^{s_1,\,s_1,\,K_0}(u\circ f)$
for some constant $K_0$ determined by \eqref{e5.12} and the constants appearing in \eqref{e1.2} for $\mu_\cy$.

Now we estimate $\|\vec h\|_{L^{n_1/s_1}(\cx,\,\ell^{n_1/s_1})}$.
In fact, since
$[L_f(x,\,2^{-j})]^{n_2}\sim |f(B(x,\,2^{-j}))|,$
by the H\"older inequality and $n_1/s_1=n_2/s_2$, we then have
$$  h_j(x)\ls 2^{js_1}\lf\{\int_{f(B(x,\,2^{-j}))}[\wz g_{L_f(f^{-1}(y),\,2^{-j})}(y)]^{n_1/s_1}\,d\mu_\cy(y)\r\}^{s_1/n_1}.
$$ Noticing that $y\in f(B(x,\,2^{-j}))$ implies that $x\in B(f^{-1}(y),\,2^{-j})$, by an argument similar to that of 
the proof of 
Theorem \ref{t1.3}, we have
\begin{eqnarray*}
&&\|\vec h\|^{n_1/s_1}_{L^{n_1/s_1}(\cx,\,\ell^{n_1/s_1}) }\\
&&\quad
\ls \int_\cy\sum_{k\in\zz}
\lf(\sharp\lf\{j\in\zz:\,  L_f(f^{-1}(y),\,2^{-j})\in[2^{-k-1},\,2^{-k})\r\}\r) [\wz g_k(y)]^{n_1/s_1} \,d\mu_\cy(y).
\end{eqnarray*}
Moreover, observe that for all $k\in\zz$ and $y\in\cy$, by \eqref{e5.12},
\begin{equation*}
 \sharp\lf\{j\in\zz:\  L_f(f^{-1}(y),\,2^{-j })\in[2^{-k-1},\,2^{-k}) \r\}\ls 1.
\end{equation*}
By $n_1/s_1=n_2/s_2$, we then have
\begin{eqnarray*}
\|\vec h\|^{n_1/s_1}_{L^{n_1/s_1}(\cx,\,\ell^{n_1/s_1})}
 \ls \sum_{k\in\zz}\int_\cy [\wz g_k(y)]^{n_1/s_1} \,d\mu(y)\ls \|\vec g\|^{n_2/s_2}_{L^{n_2/s_2}(\cy,\,\ell^{n_2/s_2})}.
\end{eqnarray*}
Thus, $\|u\circ f\|_{\dot M^{s_1}_{n_1/s_1,\,n_1/s_1}(\cx)}\ls\|u\|_{\dot M^{s_2}_{n_2/s_2,\,n_2/s_2}(\cy)}$.
Applying the above result to $f^{-1}$, we obtain that $\|u\circ f^{-1}\|_{\dot M^{s_1}_{n_1/s_1,\,n_1/s_1}(\cy)}\ls\|u\|_{\dot M^{s_2}_{n_2/s_2,\,n_2/s_2}(\cx)}$, which completes the proof of Theorem \ref{t5.1}.
\end{proof}

Moreover, combining  the proofs of
 Theorems \ref{t5.1} and \ref{t1.3}, one can further obtain the following conclusion.

\begin{cor}\label{c5.2}
Let $\cx$ and $\cy$ be Ahlfors $n_1$-regular and $n_2$-regular spaces with $n_1,\,n_2\in(0,\,\fz)$, respectively.
Let  $f$ be a quasisymmetric mapping from $\cx$ onto $\cy$,
and assume that $f$ and $f^{-1}$ are absolutely continuous and 
$J_f\in\cb_{r}(\cx)$ for some $r\in(1,\,\fz]$. 
Let $s_i\in(0,\,n_i)$ with $i=1,\,2$ satisfy  $n_1/s_1=n_2/s_2$, and $q\in(0,\,\fz]$.
Then  $f$  induces an equivalence between $\dot M^{s_1}_{n_1/s_1,\,q}(\cx)$
and $\dot M^{s_2}_{n_2/s_2,\,q}(\cy)$.
\end{cor}

In Corollary \ref{c5.2}, $f$ acts a composition operator. Moreover, 
with the assumptions of Corollary \ref{c5.2}, by Lebesgue-Radon-Nykodym Theorem and \cite{st89}, 
we have that $J_{f^{-1}}(y)=[J_f(f^{-1}(y))]^{-1}$ for almost all $y\in\cy$, 
and hence $J_{f^{-1}}\in\cb_{r'}(\cy)$ for some $r'\in(1,\,\fz]$.

\noindent Pekka Koskela

\noindent Department of Mathematics and Statistics,
P. O. Box 35 (MaD),
FI-40014, University of Jyv\"askyl\"a,
Finland
\smallskip

\noindent{\it E-mail address}:   \texttt{pkoskela@maths.jyu.fi}

\bigskip

\noindent Dachun Yang

\medskip

\noindent School of Mathematical Sciences,
 Beijing Normal University,
 Laboratory of Mathematics and Complex Systems, Ministry of Education,
Beijing 100875, People's Republic of China

\smallskip

\noindent{\it E-mail address}: \texttt{dcyang@bnu.edu.cn}

\bigskip

\noindent Yuan Zhou

\medskip

\noindent Department of Mathematics and Statistics,
P. O. Box 35 (MaD),
FI-40014, University of Jyv\"askyl\"a,
Finland

\smallskip

\noindent{\it E-mail address}:  \texttt{yuan.y.zhou@jyu.fi}

\end{document}